\documentclass{amsproc}
\usepackage{color,geometry}
\numberwithin{equation}{section}
\usepackage{mathrsfs}
\newtheorem{thm}{Theorem}[section]

\theoremstyle{definition}

\theoremstyle{remark}
\newtheorem{rem}[thm]{Remark}
\numberwithin{equation}{section}
\newcommand{\commentout}[1]{}

\newcommand{\diag}{\mathrm{diag}}

\begin{document}

\title{A study of the associated linear problem for $q$-$\mathrm{P}_{\mathrm{V}}$}

\author{Christopher M. Ormerod}

\begin{abstract}
We consider the associated linear problem for a $q$-analogue of the fifth Painlev\'e equation ($q$-$\mathrm{P}_{\mathrm{V}}$). We identify a lattice of connection preserving deformations in the space of the connection data for the linear problem with the lattice of translational B\"acklund transformations for $q$-$\mathrm{P}_{\mathrm{V}}$, hence, show all translational B\"acklund transformations possess a Lax pair. We shall show that the big $q$-Laguerre polynomials, and a suitable generalization, solve a special case of the linear problem, hence, find solutions to $q$-$\mathrm{P}_{\mathrm{V}}$ in terms of determinants of Hankel matrices with entries consisting of rational or hypergeometric functions. 
\end{abstract}

\maketitle

\section*{Introduction}

The Painlev\'e equations were isolated in the classification of second order non-autonomous differential equations with no movable singularities except for poles \cite{PainleveProperty}. For each Painlev\'e equation, there exists a system of linear ordinary differential equations with rational coefficients of the form
\begin{equation}\label{diff}
\frac{\mathrm{d}}{\mathrm{d}x}Y(x) = A(x)Y(x),
\end{equation}
such that the Painlev\'e equation arises as necessary conditions so that the associated monodromy representation is preserved \cite{FlaschkaNewell, Garnier, Jimbo:Monodromy2, Jimbo:Monodromy3, Jimbo:Monodromy1}. In addition to the unique continuous isomonodromic deformation of \eqref{diff}, Jimbo and Miwa \cite{Jimbo:Monodromy2} identify a class of discrete isomonodromic deformations that may be interpreted as translations of the monodromy data. These transformations of the linear problem, $Y(x) \to \tilde{Y}(x)$, are induced by left multiplication with a rational matrix, $R(x)$, via
\begin{equation}\label{transformation}
\tilde{Y}(x) = R(x)Y(x).
\end{equation}
The transformed matrix, $\tilde{Y}(x)$, solves
\begin{equation}\label{tilded}
\frac{\mathrm{d}}{\mathrm{d}x}\tilde{Y}(x) = \tilde{A}(x)\tilde{Y}(x),
\end{equation}
where
\begin{equation}\label{diffBacklund}
\tilde{A}(x) = \left( \frac{\mathrm{d}R(x)}{\mathrm{d}x} R^{-1}(x) + R(x)A(x)R(x)^{-1}\right).
\end{equation}
If $\tilde{A}(x)$ is of the same fundamental form as $A(x)$, then the unique continuous isomonodromic deformation of \eqref{tilded} specifies that the entries of $\tilde{A}(x)$ are also solutions to the same Painlev\'e equation, with a possible change in parameters governed by $\tilde{A}(x)$'s monodromy data \cite{Jimbo:Monodromy2}. Therefore, we consider the transformation, \eqref{transformation}, as inducing a B\"acklund transformation. We may consider this transformation as a system of difference equations in the ``tilde'' direction, hence, the resulting B\"acklund transformation may be thought of as the compatibility of a differential-difference system. Jimbo and Miwa state that all the B\"acklund transformations for the Painlev\'e equations arise in this manner \cite{Jimbo:Monodromy2}.

The discrete Painlev\'e equations are second order non-autonomous difference equations that admit the Painlev\'e equations in a continuum limit \cite{Gramani:DiscretePs}. There are three types of discrete Painlev\'e equation; additive, multiplicative and elliptic, which are classified  according to how the parameters evolve \cite{Sakai:Rational}. The discrete Painlev\'e equations are integrable in the sense that they possess a discrete analogue of the Painlev\'e property, known as singularity confinement \cite{Gramani:DPainleveProperty}, and many may be solved via associated linear problems \cite{Hay, Sakai:qP6, Murata2009, Gramani:Isomonodromic}.

Some of the additive discrete Painlev\'e equations arise as translational components of the continuous Painlev\'e equations, hence, \eqref{diff} and \eqref{transformation} give rise to differential difference Lax-pairs for known Lax pairs for discrete Painlev\'e equations \cite{Jimbo:Monodromy2, Jimbo:Monodromy3, Jimbo:Monodromy1}. For discrete Painlev\'e equations of the multiplicative type, where the parameters change by some multiplicative constant, $q \in \mathbb{C}\setminus \{0\}$ and $|q| \neq 1$, the associated linear problems are systems of linear $q$-difference equations \cite{Hay, Sakai:qP6, Murata2009, Gramani:Isomonodromic}. 

For systems of linear $q$-difference equations, represented by
\begin{equation}\label{qdiff}
Y(q x) = A(x)Y(x),
\end{equation}
where $A(x)$ is a rational $m\times m$ matrix, we may associate a connection matrix \cite{Adams, qBirkhoff, Carmichael, Sauloy, vanderPut}. A $q$-analogue of the sixth Painlev\'e equation was found to be equivalent to the conditions necessary for a certain deformation of a system, of the form \eqref{qdiff}, to preserve the associated connection matrix \cite{Sakai:qP6}. This condition is that deformed system is related via \eqref{transformation}, where $\tilde{Y}(x)$ satisfies 
\begin{equation}
\tilde{Y}(qx) = \tilde{A}(x)\tilde{Y}(x),
\end{equation}
and
\[
A(x) \mapsto R(qx)A(x)R(x)^{-1}.
\]
In light of Sakai's framework \cite{Sakai:Rational}, the discrete Painlev\'e equations represent just one translational component of the group of B\"acklund transformations. One should expect that other translational B\"acklund transformations are expressible as connection preserving deformations. The aim of this article is to develop this idea in the context of the associated linear problem for $q$-$\mathrm{P}_{\mathrm{V}}$ found by Murata \cite{Murata2009}. The version of $q$-$\mathrm{P}_{\mathrm{V}}$ chosen is a translational component of a birational representation of an extended affine Weyl group of type $A_4^{(1)}$, which we denote $T_4$, and is given by
\begin{equation}
\label{qPV:geo} T_4 : \left\{ \begin{array}{c  } b_0, b_1 ,b_2\\ b_3, b_4 \end{array} ; f, g \right\} \to \left\{ \begin{array}{c  } \frac{b_0}{q}, b_1 ,b_2\\ b_3, q b_4 \end{array} ; \tilde{f}, \tilde{g} \right\},
\end{equation}
where $q = 1/ b_0b_1 b_2b_3b_4$ and $\tilde{f}$ and $\tilde{g}$ are related to $f$ and $g$ via
\begin{subequations}\label{intro:qP5}
\begin{align}
\label{intro:ftilde}\tilde{f}f =& \frac{b_1b_2}{b_3}\frac{1-g}{(1+b_2g)(1-b_1b_2g)},\\
\label{intro:gtilde}\tilde{g}g =& \frac{b_4b_0}{b_2}\frac{(1-b_3\tilde{f})(1-\tilde{f})}{\tilde{f}(b_4 - q\tilde{f})}.
\end{align}
\end{subequations}
We intend to show that this system is just one example of the group of connection preserving deformations. In doing so, we shall provide a basis for the lattice of connection preserving deformations, and hence, factorize the connection preserving deformation that defines the evolution of $q$-$\mathrm{P}_{\mathrm{V}}$ into more elementary connection preserving deformations. 

Just as one may consider special solutions of the $q$-Painlev\'e equations, we shall consider a special solution of the associated linear problem. It is the second aim of this paper to use techniques recently applied to $q$-$\mathrm{P}_{\mathrm{VI}}$ \cite{OrmerodForresterWitte} to show that the big $q$-Laguerre polynomials satisfy form a vector solution to the associated linear problem for $q$-$\mathrm{P}_{\mathrm{V}}$. Finding this particular solution of the associated linear problem gives rise to a family of special solutions of $q$-$\mathrm{P}_{\mathrm{V}}$, which may be expressed in terms of variables associated with the orthogonal polynomial framework \cite{OrmerodForresterWitte}. These results bear a remarkable similarity to the work of Masuda, who, for a different version of $q$-$\mathrm{P}_\mathrm{V}$, forms determinental type solutions where the entries of the relevant matrix are continuous $q$-Laguerre polynomials \cite{Masuda}. However, this work adopts a significantly different approach to \cite{Masuda}, as this approach is entirely based on the associated linear problem alone.

\section{$q$-calculus and $q$-special functions}

We first introduce some preliminary material before proceeding. The first set of theory we require is the notion of a $q$-calculus, which is an difference analogue of regular calculus. For an extensive reference on the $q$-calculus, see \cite{Quantum}. Just as in the introduction, we fix a $q \in \mathbb{C} \setminus \{0\}$, where we will make the additional assumption that $|q| < 1$. This assumption holds for the remainder of the paper. We consider the $q$-difference operator
\begin{equation}\label{q:diff}
D_{q,x} f(x) = \frac{f(x) - f(qx)}{x(1-q)}.
\end{equation}
The $q$-differential analogues of the multiplication and quotient rule are
\begin{subequations}
\begin{align}
\label{qmultrule}D_{q,x} f(x)g(x) &= f(x)D_{q,x}g(x) + g(qx)D_{q,x}f(x) ,\\
&= g(x)D_{q,x}f(x) - f(qx)D_{q,x}g(x),\nonumber \\ 
\label{qquotrule}D_{q,x} \frac{f(x)}{g(x)} &= \frac{g(x)D_{q,x}f(x) - f(x)D_{q,x}g(x)}{g(x)g(qx)}.
\end{align}
\end{subequations}
If $f$ is continuously differentiable, then
\[
\lim_{q\to 1} D_{q,x} f(x) = \frac{\mathrm{d}}{\mathrm{d}x} f(x)
\]
Associated with the $q$-derivative is the Jackson $q$-integral \cite{Thomae1}, written as
\begin{equation}\label{ortho:Jacksons}
\int_{0}^{z} f(t)\mathrm{d}_q t := z(1-q)\sum_{n=0}^{\infty} f(zq^n)q^n.
\end{equation}
More generally, we write
\[
\int_{a}^{b} f(t)\mathrm{d}_q t = \int_{0}^{a} f(t)\mathrm{d}_q t - \int_{0}^{b} f(t)\mathrm{d}_q t.
\]
We define the $q$-analogue of the Pochhammer symbol as
\begin{subequations}
\begin{align}
\label{q:Pochhammer0} (a;q)_0 &= 1,\\
\label{q:Pochhammerk} (a;q)_k &= \prod_{n=0}^{k-1} (1-aq^{n}),\\
\label{q:Pochhammerinfty} (a;q)_\infty &= \prod_{n=0}^{\infty} (1-aq^{n}).
\end{align}
It is also convenient to define the notation
\begin{equation}
\label{q:Pochhammersequence}(a_1,\ldots,a_n;q)_k = \prod_{m=1}^{n} (a_m;q)_k.
\end{equation}
\end{subequations}

We define the generalized $q$-hypergeometric function as
\begin{equation}\label{ortho:Hypergeometric}
{ }_r \phi_s \left( \begin{array}{ c |} a_1, \ldots, a_r \\ b_1, \ldots, b_s \end{array} \hspace{.1cm} q ; z \right) = \sum_{k=0}^{\infty} \frac{(a_1,\ldots, a_r;q)_k}{(b_1,\ldots, b_s,q;q)_k} (-1)^{(1+s-r)k}q^{(1+s-r){k \choose 2}} z^k.
\end{equation}
which is the generalization of the basic hypergeometric function, or Heine's hypergeometric series, given by
\[
{}_2 \phi_1 \left( \begin{array}{ c |} a, b \\ c \end{array} \hspace{.1cm} q; z \right) = \sum_{k=0}^{\infty} \frac{(a,b;q)_k}{(c,q;q)_k} z^k.
\]
The basic hypergeometric function possesses the following Jackson $q$-integral representation \cite{GasperRahman} 
\begin{equation}\label{hyperint}
{}_2\phi_1 \left(\begin{array}{c |} a , b \\ c \end{array} \hspace{.1cm} q;t\right) = \frac{\left(b, \frac{c}{b};q\right)_\infty}{(1-q)(c,q;q)_\infty}\int_0^1 x^{\log_q b -1} \frac{(xta,xq;q)_\infty}{\left(xt,\frac{xc}{b};q\right)_\infty} \mathrm{d}_q x.
\end{equation}

The famous Jacobi triple product identity \cite{Jacobi69fundamentanova} defines $\theta_q$ to be 
\begin{align*}
\theta_q (x) = \left(-q x,-\frac{1}{x},q;q\right)_\infty = \sum_{k=-\infty}^{\infty} x^k q^{{k \choose 2}},
\end{align*}
which satisfies
\[
\theta_q(q x) = qx \theta_q(x).
\]
We also define the $q$-character to be
\[
e_{q,c} (x) = \frac{\theta_q \left( x \right) \theta_q \left( \frac{1}{c}\right)}{\theta_q \left(\frac{x}{c} \right)},
\]
which satisfies the equations
\begin{align*}
e_{q,c} (qx) &= c e_{q,c}(x),\\
e_{q,qc} (x) &= x e_{q,c}(x).
\end{align*}
These functions are useful in the classification of systems of $q$-difference equations \cite{Sauloy, vanderPut}. 

\section{Connection preserving deformations}

Given a system of the form \eqref{qdiff}, we may consider the solution of a multiple of $Y(x)$ by $q$-characters and $q$-Pochhammer symbols that reduce \eqref{qdiff} to a problem in which $A(x)$ is a polynomial matrix, given by
\[
A(x) = A_0 + A_1x + \ldots + A_mx^m.
\]
We shall assume that $A_0$ and $A_m$ are diagonalizable $N\times N$ matrices. We have two series solutions, $Y_0(x)$ and $Y_\infty(x)$, around $x=0$ and $x= \infty$ respectively. According to the theory of Birkhoff and Carmichael \cite{qBirkhoff, Carmichael}, the series solution around $x=0$ ($x= \infty$) converge if the ratio of the eigenvalues of $A_0$ ($A_m$) are not powers of strictly positive powers of $q$ \cite{qBirkhoff, Carmichael}. This case is known as the regular case, and covers the associated linear problem for $q$-$\mathrm{P}_\mathrm{VI}$ \cite{Sakai:qP6}. This regular case mimics the theory concerning isomonodromic deformations of Fuchsian systems of linear differential equations \cite{Fuchs2, Fuchs1}.

Under the more general theory of Adams \cite{Adams}, these restrictions on the eigenvalues are completely relaxed. One obtains a more general series solution for $Y_0(x)$ and $Y_\infty(x)$. In this paper, we will assume an expansion of the form
\begin{subequations}
\begin{align}
Y_0(x) &= \hat{Y}_0(x) \diag \left( \theta_q\left(\frac{x}{q}\right)^{l_i} \lambda_i \right),\\
Y_\infty(x) &= \hat{Y}_\infty(x) \diag \left( \theta_q\left(\frac{x}{q}\right)^{k_i} \kappa_i \right),
\end{align}
\end{subequations}
where $\hat{Y}_0(x)$ and $\hat{Y}_{\infty}(x)$ represent some series in $x$ around $x=0$ and $x=\infty$ respectively and the $l_i$ and $k_i$ are fixed integers. This expansion does not assume that the eigenvalues of the leading terms in the expansion around $x=\infty$ or $x = 0$ are all non-zero. It is just one step in the full generalization of the series solution expansions of \eqref{qdiff} around $x=0$ and $x= \infty$.

Just as in the theory of monodromy, we consider the series solution around $x= \infty$, $Y_\infty(x)$, as being fundamental in the sense that we may normalize $Y_\infty(x)$ and use a connection matrix to express all other solutions in terms of $Y_\infty(x)$. We may express $Y_0(x)$ in terms of $Y_\infty(x)$, thus, form the connection matrix, given by 
\[
Y_0(x) = Y_\infty(x) P(x),
\]
where $P(x)$ is quasi-periodic in $x$ by definition. It is known that $P(x)$ is a matrix over the field of functions spanned by elements of the
\[
\phi_{c,d,q}(x) = \frac{e_{q,c}(x) e_{q,d}(x)}{e_{q,cd}(x)}
\]
where $c$ and $d$ are constants \cite{Sauloy, vanderPut}. We note that $Y_\infty(x)$ and $Y_0(x)$ are related to the symbolic solutions
\begin{align*}
Y_0(x) &\sim A(x/q) A(x/q^2) \ldots,\\
Y_\infty (x) &\sim A(x)^{-1} A(qx)^{-1} \ldots,
\end{align*}
where $\sim$ denotes an equivalence modulo the conjugation of some set of transformations so that the multiplicative series converges. We parameterize the determinant of $A(x)$ by writing
\[
\det(A(x)) = \kappa \prod_i (x-a_i)
\]
where $\kappa = \prod_i \kappa_i$. From the above expression, $Y_0(x)$ is possibly singular at $q^ja_i$ for all $i$ and $j > 0$. In a similar manner, we may not be able to define $Y_{\infty}(a_i/q^j)^{-1}$ for all $i$ and $j \geq 0$. Hence, $P(x)$ is possibly singular at $\{ q^n a_i \}$. We consider the connection data, $M$, to be 
\begin{equation}\label{connectiodata}
M = \left\{ \begin{array}{c}  a_i \\ \lambda_i, \kappa_i \end{array} \right\}.
\end{equation}
Rather than considering the only connection preserving deformation of interest to be the one that defines the discrete Painlev\'e equation, we wish to explore the space of connection preserving deformations. We propose that associated linear problems possess a set of connection preserving deformations that have some structure that goes hand-in-hand with the B\"acklund transformations.

We now consider transformations of the form \eqref{transformation}. Note that if $R(x)$ is some rational invertible matrix, then
\begin{align*}
\tilde{Y}(qx) &= R(qx)Y(qx),\\
&= R(qx) A(x) Y(x),\\
&= R(qx) A(x) R(x)^{-1} R(x)Y(x),\\
&= \tilde{A}(x) \tilde{Y}(x).
\end{align*}
where
\[
R(qx) A(x) R(x)^{-1} = \tilde{A}(x).
\]
Furthermore, 
\begin{align*}
\tilde{P}(x) &= \tilde{Y}_{\infty}(x)^{-1}\tilde{Y}_0(x), \\
&= Y_\infty(x)^{-1} R(x)^{-1} R(x)Y_0(x) = P(x),
\end{align*}
hence, $A(x)$ and $\tilde{A}(x)$ possess the same connection matrix. The compatibility can be seen when one tries to compute $\tilde{Y}(qx)$, as we require $\tilde{A}(x)\tilde{Y}(x) = R(qx)Y(qx)$, imposing the constraint \cite{Sakai:qP6, Murata2009, Gramani:Isomonodromic, Sakai:Garnier}
\begin{equation}\label{P5:comp}
\tilde{A}(x)R(x) = R(qx) A(x).
\end{equation}
We now endow the set of transformations with some group structure. Take two matrices, $R_1$ and $R_2$, inducing transformations
\begin{align*}
\tilde{Y}(x) &= R_1(x)Y(x),\\
\hat{Y}(x) &= R_2(x)Y(x),
\end{align*}
then,
\begin{equation}\label{multrule}
\hat{\tilde{Y}}(x) = \hat{R}_1(x) R_2(x) Y(x) = \hat{R}_1(x) R_2(x) Y(x).
\end{equation}
The composition of these two connection preserving deformations is represented by \eqref{transformation} where $R(x) = \hat{R}_1(x) R_2(x)$, which, by consistency, must also be $\hat{R}_1(x) R_2(x)$. 

To determine the form of $R(x)$, we are required to examine the determinental constraint. By taking  determinants of \eqref{qdiff}, one finds that $\det Y(x)$ satisfies the scalar equation
\[
\det Y(qx) = \det A(x) \det Y(x).
\]
Using the functions specified in \S 1, it is clear that we may solve this in terms of $q$-exponentials and theta functions. If we know how the connection data changes this determinant, then $\det R(x)$ is a rational solution of
\begin{equation}\label{detR}
\det R(x) = \frac{\det \tilde{Y}}{\det Y}.
\end{equation}
Further information regarding the asymptotics of $R$ may be inferred from
\begin{align*}
R(x) &= \tilde{Y}_{\infty}(x) Y_{\infty}(x)^{-1},\\
 &= \tilde{Y}_{0}(x) Y_{0}(x)^{-1},
 \end{align*}
where use may be made of the known asymptotic forms and the changes to the connection data expected from the connection preserving deformation.

\subsection{Parameterization of $q$-$\mathrm{P}_{\mathrm{V}}$} 

The information one is required to specify to obtain $q$-$\mathrm{P}_{\mathrm{V}}$ is minimal \cite{Murata2009}. From the connection preserving deformation theory, we take a linear problem of the form \eqref{qdiff} where $A$ is a $2\times 2$ polynomial matrix, which we parameterize by letting
\[
A(x) = A_0 + A_1 x + A_2x^2 = (a_{ij}(x))_{i,j=1,2}.
\]
We specify the following properties:
\begin{itemize}
\item{The determinant of $A$ is
\begin{equation}\label{detA}
\det A(x) = \kappa_1\kappa_2(x-a_1)(x-a_2)(x-a_3).
\end{equation}
}
\item{The solution at $x=\infty$ is specified by
\begin{equation}\label{expansionYinf}
Y_\infty(x) = \left( I + \sum_{i> 0} Y_\infty^{(i)}\frac{1}{x^i} \right)\begin{pmatrix} \theta\left(\frac{x}{q}\right) e_{q,\kappa_1}(x) & 0 \\ 0 & \theta\left(\frac{x}{q}\right)^2 e_{q,\kappa_2}(x)\end{pmatrix}.
\end{equation}
}
\item{The solution at $x=0$ is specified by
\begin{equation}\label{expansionY0}
Y_0(x) = \left( C_0 +  \sum_{i> 0} Y_0^{(i)}x^i\right) \begin{pmatrix} e_{q,\lambda_1}(x) & 0 \\ 0 & e_{q,\lambda_2}(x) \end{pmatrix},
\end{equation}
where $C_0$ diagonalizes $A(x)$ at $x=0$.
}
\end{itemize}
We note that the series solution around $0$ is handled by the theory for the regular case \cite{qBirkhoff, Carmichael, Sakai:qP6}, and the form of the solution around $\infty$ is of particular interest, as it deviates from the regular case, and is handled by the theory for irregular case \cite{Adams}. It should be noted that the above asymptotic expansions around $x= \infty$ seems valid for the associated linear problem for cases below $q$-$\mathrm{P}_{\mathrm{V}}$ in Muratas' work \cite{Murata2009}, however, the expansion around $x=0$ for $q$-$\mathrm{P}_\mathrm{V}$ is different to those systems below $q$-$\mathrm{P}_\mathrm{V}$ in the hierarchy. The expansion of solutions for Jimbo and Sakai's associated linear problem for $q$-$\mathrm{P}_{\mathrm{VI}}$ \cite{Sakai:qP6} is different in that we must replace the single factor of $\theta(x/q)$ in \eqref{expansionYinf} with $\theta(x/q)^2$. 

As a $2\times 2$ system with the above properties, this specifies that one of the eigenvalues of $A_2$ is $0$, and the other is $\kappa_2$, meaning that we may specify
\[
A_2 = \begin{pmatrix} 0 & 0 \\ 0 & \kappa_2 \end{pmatrix}.
\]
This forces the top left entry of $A_1$ to be $\kappa_1$. This also constrains the eigenvalues of $A_0$, $\lambda_1$ and $\lambda_2$, to be non-zero, with the additional constraint that
\[
\lambda_1\lambda_2 = -\kappa_1\kappa_2 a_1 a_2 a_3.
\]
This specifies that the connection data is given by
\[
M = \left\{\begin{array}{c c c c} a_1 & a_2 & a_3 & \\ \kappa_1 & \kappa_2 & \lambda_1 & \lambda_2 \end{array}\right\},
\]
with the implicit relation above.

This associated linear problem must be parameterizable in terms of three free variables, $y$, $z$ and $w$. Following the work of previous authors \cite{Murata2009, Sakai:qP6}, the three parameter parameterization is completely specified by the conditions 
\begin{align*}
a_{12}(x) &= w(x-y),\\
a_{11}(y) &= \kappa_1z_1,\\
a_{22}(y) &= \kappa_2z_2,
\end{align*}
where $w$ is a parameter that encapsulates the gauge freedom of the system. Evaluating the determinant at $x= y$ reveals
\[
z_1 z_2 = (y-a_1)(y-a_2)(y-a_3).
\]

We remark that this type of parameterization a theme that arises in the associated linear problems for the associated linear problems for the continuous Painleve equations, $\mathrm{P}_{\mathrm{I}\textrm{-}\mathrm{VI}}$, as listed in the work of Jimbo et al. \cite{Jimbo:Monodromy2} \footnote{Where the parameterization of the associated linear problem for $\mathrm{P}_{\mathrm{I}}$ appears to an exception, however, $0$ of the bottom left entry is an equally valid choice.}. This theme is followed in the parameterization for the associated linear problems for the $q$-Painlev\'e equations \cite{Sakai:qP6, Murata2009}. A general form for this parameterization for our system is
\begin{equation}\label{rep}
A(x) = \begin{pmatrix} \kappa_1(x-y + z_1) & \kappa_2w(x-y) \\
\frac{\kappa_1 (\gamma x + \delta)}{w} & \kappa_2 ((x-\alpha)(x-y) + z_2)
\end{pmatrix}.
\end{equation}
We know 
\[
\mathrm{trace} A(0) = \lambda_1 + \lambda_2
\]
is a linear in $\alpha$, from which we obtain
\begin{subequations}\label{abvals}
\begin{equation}
\alpha = \frac{y \kappa _1-z_1 \kappa _1-z_1 \kappa _2+\lambda _1+\lambda _2}{y
   \kappa _2}.
\end{equation}
Equating the coefficient of $x^2$ from \eqref{rep} with the coefficient of $x^2$ in \eqref{detA} gives
\begin{equation}
\gamma = a_1+a_2+a_3-2 y+z_1-\alpha.
\end{equation}
Equating the $\det A(0)$ from \eqref{rep} with $\lambda_1\lambda_2$ gives
\begin{equation}
\delta = \frac{\kappa_1\kappa_2\left(y-z_1\right) \left(y \alpha +z_1\right)+\lambda _1 \lambda _2}{y \kappa_1\kappa_2}.
\end{equation}
\end{subequations}
Finally, equating $\det A(y)$ from \eqref{rep} with the case when $x=y$ in \eqref{detA} gives
\begin{subequations}\label{zvals}
\begin{align}
z_1 &= \frac{\left(a_1-y\right) \left(a_2-y\right)}{z},\\
z_2 &= \left(y-a_3\right) z.
\end{align}
\end{subequations}
which represents a different parameterization from Murata \cite{Murata2009}. The reasons for this will become apparent later.

Computing the first few $Y_\infty^{(i)}$ and $Y_0^{(i)}$ in \eqref{expansionYinf} and \eqref{expansionY0} reveals
\begin{align}\label{firstterms}
Y_\infty(x) &= \left(\begin{pmatrix} 1 & 0 \\ 0 & 1 \end{pmatrix} + \frac{1}{x} \begin{pmatrix} \frac{q \left(y-z_1+\gamma \right)}{q-1} & q w \\
 -\frac{\gamma  \kappa _1}{w \kappa _2} & \frac{q (y+\alpha )}{q-1} \end{pmatrix} + O\left(\frac{1}{x^2}\right) \right)\begin{pmatrix} \theta\left(\frac{x}{q}\right) e_{q,\kappa_1}(x) & 0 \\ 0 & \theta\left(\frac{x}{q}\right)^2 e_{q,\kappa_2}(x)\end{pmatrix},\\
Y_0(x) &= \left( \begin{pmatrix} \frac{1}{\lambda_1-\lambda_2} & \kappa_2wy \\ \frac{\kappa_1(z_1-y)-\lambda_1}{\kappa_2wy(\lambda_1-\lambda_2)} & \kappa_1(z_1-y) - \lambda_2 \end{pmatrix} + O(x) \right) \begin{pmatrix} e_{q,\lambda_1}(x) & 0 \\ 0 & e_{q,\lambda_1}(x) \end{pmatrix}.
\end{align}
Further terms may be calculated, however, it is difficult to write them in a succinct form.

We are now able to specify which connection preserving deformation gives rise to $q$-$\mathrm{P}_{\mathrm{V}}$
\begin{equation}\label{P5:transformation}
\left\{ \begin{array}{c c c c} a_1 & a_2 & a_3 & \\ \kappa_1 & \kappa_2 & \lambda_1 & \lambda_2 \end{array} ; y,z \right\} \to \left\{ \begin{array}{c c c c} qa_1 & qa_2 & a_3 & \\ \frac{\kappa_1}{q} & \frac{\kappa_2}{q} & \lambda_1 & \lambda_2 \end{array} \tilde{y},\tilde{z}\right\}.
\end{equation}
When we consider this as being induced by a transformation of the form \eqref{transformation}. Since $\kappa_1$ and $\kappa_2$ changed in the transformation, using \eqref{expansionYinf}, 
\begin{equation}\label{firstterm}
R(x) = \tilde{Y}(x)Y(x)^{-1} \sim \frac{1}{x}I
\end{equation}
as $x \to \infty$, and since the $\lambda_1$ and $\lambda_2$ are unchanged in the transformation, 
\[
R(x) = \tilde{Y}(x)Y(x)^{-1} \sim I
\]
as $x \to 0$. We also have that
\[
\det R(x) = \frac{1}{(x-qa_1)(x-qa_2)},
\]
hence, $R(x)$ is a matrix of the form
\begin{equation}\label{Rmat}
R(x) = \frac{xI+ R_0}{(x-q a_1)(x-qa_2)}.
\end{equation}
We parameterize $R_0$ by letting $R_0 = (r_{ij})_{i,j=1,2}$. We shall see that \eqref{P5:comp} specifies both the entries of $R_0$, and the evolution of the entries of $A$ in the tilde direction.

\begin{thm}\label{thm:qP5}
The transformation of connection data given by \eqref{P5:transformation} induces the following rational transformation
\begin{subequations}
\begin{align}
\label{P5:what}\frac{\tilde{w}}{w}&=\frac{q \kappa_1}{\kappa_2} \frac{1}{{\displaystyle \frac{q\kappa _1 }{\kappa_2}}-\tilde{z}},\\
\label{P5:yhat}\tilde{y}y&=a_3\frac{\left(\tilde{z}_n +{\displaystyle \frac{q\lambda_1}{\kappa_2a_3}}\right) \left(\tilde{z}_n+ {\displaystyle\frac{q\lambda_2}{\kappa_2a_3}}\right)}{  \left(\tilde{z} - {\displaystyle\frac{q\kappa_1}{\kappa_2}}\right)},\\
\label{P5:zhat}\tilde{z}z&=\frac{q\kappa_1}{\kappa_2}\frac{\left(y-a_1\right) \left(y-a_2\right)}{\left(y-a_3\right)}.
\end{align}
\end{subequations}
\end{thm}

\begin{proof}
The simplest way of finding $R_0$ is to determine what values of $r_{ij}$ are required for the leading non-vanishing terms in \eqref{P5:comp} to be zero. This gives us the matrix
\begin{equation}\label{P5:B0}
R_0 = \begin{pmatrix}
\displaystyle \frac{q(\tilde{y}-y+ \tilde{\alpha} - \alpha)}{1-q} & q(\tilde{w}-w) \\ {\displaystyle \frac{\kappa_1}{\kappa_2}\left( \frac{\gamma}{w} - \frac{\tilde{\gamma}}{\tilde{w}}\right)} & \displaystyle-\frac{q(\tilde{y}-y +\tilde{\alpha}- \alpha)}{1-q} + qa_1+qa_2
\end{pmatrix}.
\end{equation}
By equating the residue of the upper right entries of the left and right hand sides of \eqref{P5:comp} at $x=qa_1$ and $x=qa_2$ is equivalent to
\begin{align*}
\frac{\kappa _2 \left(a_1 q+b_{22}\right) \tilde{w} \left(a_1 q-\tilde{y}\right)+b_{12} \kappa _1
   \left(qa_1 -\tilde{y}+\tilde{z}_1\right)}{\left(a_1-a_2\right) q^2}&= 0, \\
\frac{\kappa _2 \left(a_2 q+b_{22}\right) \tilde{w} \left(\tilde{y}-a_2 q\right)+b_{12} \kappa _1
   \left(-qa_2+\tilde{y}-\tilde{z}_1\right)}{\left(a_1-a_2\right) q^2}&= 0,
\end{align*}
in which we may deduce that $r_{12}$ and $r_{22}$ are
\begin{align*}
r_{12}&=\frac{\kappa _2 \tilde{w} \left(\tilde{y}-q a_1\right) \left(\tilde{y}-q a_2\right)}{\kappa_1 \tilde{z}_1},\\
r_{22}&=-\frac{\left(\tilde{y}-q a_1\right) \left(\tilde{y}-q a_2\right)}{\tilde{z}_1}-q \left(a_1+a_2\right)+\tilde{y}.
\end{align*}
Equating the representation of $r_{12}$ directly above with that of \eqref{P5:B0} gives \eqref{P5:what}. Similarly, equating the residues of the upper right entry of \eqref{P5:comp} at $x = a_1$ and $x= a_2$ with $0$ gives
\begin{align*}
r_{12} = -\frac{q w \left(y-a_1\right) \left(y-a_2\right)}{\left(y-a_1\right) \left(y-a_2\right)-z_2}.
\end{align*}
The compatibility of this with \eqref{P5:B0} gives \eqref{P5:zhat}. The leading non-vanishing asymptotics of the lower right entry of \eqref{P5:comp} reveals
\[
r_{22}=\frac{q \left(\tilde{y}+\tilde{\alpha}-y-\alpha+\left(a_1+a_2\right) (1-q)\right)}{q-1}.
\]
A comparison of this and \eqref{P5:B0} gives \eqref{P5:yhat}.
\end{proof}

An important tool in the work of Jimbo and Miwa in the continuous isomonodromic deformation theory \cite{Jimbo:Monodromy2} was to use the expansions of the fundamental solutions to obtain the relevant Lax pair. We wish to extend this method to systems of $q$-difference equations. 
We present another way to obtain $R_0$, we extend the expansion of \eqref{Rmat} around $x= \infty$ to reveal
\[
\begin{pmatrix}
\displaystyle \frac{1}{x}+\frac{q a_1+q a_2+r_{11}}{x^2} & \displaystyle \frac{r_{21}}{x^2} \\
\displaystyle  \frac{r_{21}}{x^2} & \displaystyle \frac{1}{x}+\frac{q a_1+q a_2+r_{22}}{x^2}
 \end{pmatrix} + O\left( \frac{1}{x^3}\right),
\]
then compare this expansion with the large $x$ asymptotics of $\tilde{Y}(x) Y^{-1}(x)$ using \eqref{firstterms}, given by
\[
\begin{pmatrix}
 \displaystyle  \frac{1}{x}+\frac{q a_1+ qa_2}{x^2}+\frac{q (y-\tilde{y}+\alpha-\tilde{\alpha} )}{(q-1) x^2} & \displaystyle \frac{q (\tilde{w}-w)}{x^2} \\
\displaystyle  \frac{(\tilde{w} \gamma -w \tilde{\gamma}) \kappa _1}{w \tilde{w} \kappa_2 x^2} & \displaystyle  \frac{1}{x}+\frac{q (\tilde{y}-y+\tilde{\alpha}-\alpha  )}{(q-1) x^2}
\end{pmatrix} + O\left( \frac{1}{x^3}\right),
\]
which, somewhat remarkably, recovers \eqref{P5:B0}. This new method is one that we will exploit in the coming section.

We may identify \eqref{P5:yhat} and \eqref{P5:zhat} with the evolution of \eqref{intro:qP5} by letting
\begin{subequations}
\begin{align}
y &= -\frac{\lambda_1g}{a_3\kappa_2},\\
z &= -\frac{\lambda_1}{a_3\kappa_2f},
\end{align}
with a correspondence between parameters given by
\begin{equation}
b_0 = \frac{a_3}{a_1}, \,\, b_1 = \frac{a_1}{a_2}, \,\, b_2 = - \frac{a_2\kappa_2}{\lambda_2}, \,\, b_3 = \frac{\lambda_2}{\lambda_1}, \,\, b_4 =  -\frac{\lambda _1}{a_3 q \kappa _1}.
\end{equation}
\end{subequations}

We would like to remark that this choice of connection preserving deformation, where the $\kappa_1, \kappa_2 \to \kappa_1/q, \kappa_2/q$ , is different to that of Murata \cite{Murata2009}, where the eigenvalues, $\lambda_1,\lambda_2 \to q\lambda_1, q\lambda_2$. However, the birational transformation given by \cite{Murata2009} is equivalent to \eqref{P5:yhat} and \eqref{P5:zhat}, hence, we note 
\[
\left\{ \begin{array}{c c c c} a_1 & a_2 & a_3 & \\ \kappa_1 & \kappa_2 & \lambda_1 & \lambda_2 \end{array} ; y,z \right\} \to \left\{ \begin{array}{c c c c} a_1 & a_2 & a_3 & \\ q\kappa_1 & q\kappa_2 & q\lambda_1 & q\lambda_2 \end{array} y,z\right\},
\]
is non-trivial translation of the connection data but represents a trivial element of the B\"acklund transformations on both the set of parameters and the variables $y$ and $z$. 

\subsection{Translations}

It is at this point we depart from the theory most notably established by the previous authors \cite{Jimbo:Monodromy2, Sakai:qP6, Murata2009}. Although connection preserving deformations are well established, previous authors have not fully explored cases of connection preserving deformations apart from those that define the relevant discrete Painlev\'e equation. It is the task of this section to show all translational B\"acklund transformations are induced, on the level of the linear system, by transformations of the form \eqref{transformation}. The translations of the connection data are generated by elements of the form
\begin{align*}
T_{a_i,\lambda_j} : \left\{ \begin{array}{c c c c}  a_1 & a_2 & a_3 & \\ \kappa_1 & \kappa_2 & \lambda_1 & \lambda_2 \end{array}:y,z\right\} &\to \left\{ \begin{array}{c c c c}q^{\delta_{1i}} a_1 & q^{\delta_{2i}}a_2 & q^{\delta_{3i}}a_3   \\ \kappa_1 & \kappa_2 & q^{\delta_{1j}}\lambda_1 & q^{\delta_{2j}}\lambda_2 \end{array} : \tilde{y},\tilde{z} \right\},\\
T_{\kappa_i,\lambda_j} : \left\{ \begin{array}{c c c c} a_1 & a_2 & a_3 & \\ \kappa_1 & \kappa_2 & \lambda_1 & \lambda_2 \end{array}:y,z\right\} &\to \left\{ \begin{array}{c c c c} a_1 & a_2 & a_3 & \\ q^{\delta_{1i}}\kappa_1 & q^{\delta_{2i}}\kappa_2 & q^{\delta_{1j}}\lambda_1 & q^{\delta_{2j}}\lambda_2 \end{array} : \tilde{y},\tilde{z} \right\}.
\end{align*}
We shall denote the matrices that induce these actions by $R_{a_i,\lambda_j}(x)$ and $R_{\kappa_i,\lambda_j}(x)$.

We first start with a set of symmetries of $A(x)$, which must possess the same connection matrix in a trivial manner. These symmetries may be seen to be induced by the matrix $R(x) = I$. These are given by
\begin{align*}
r_0 : 
\left\{ \begin{array}{c c c c} a_1 & a_2 & a_3 & \\ \kappa_1 & \kappa_2 & \lambda_1 & \lambda_2 \end{array} : y,z \right\} &\to \left\{ \begin{array}{c c c c} a_2 & a_1 & a_3 & \\ \kappa_1 & \kappa_2 & \lambda_1 & \lambda_2 \end{array}:y,z \right\}, \\
r_1 : 
\left\{ \begin{array}{c c c c} a_1 & a_2 & a_3 & \\ \kappa_1 & \kappa_2 & \lambda_1 & \lambda_2 \end{array}:y,z\right\} &\to \left\{ \begin{array}{c c c c} a_1 & a_3 & a_2 & \\ \kappa_1 & \kappa_2 & \lambda_1 & \lambda_2 \end{array} : y,z\frac{y-a_3}{y-a_2} \right\}, \\
r_2 : 
\left\{ \begin{array}{c c c c} a_1 & a_2 & a_3 & \\ \kappa_1 & \kappa_2 & \lambda_1 & \lambda_2 \end{array}:y,z \right\} &\to \left\{ \begin{array}{c c c c} a_1 & a_2 & a_3 & \\ \kappa_1 & \kappa_2 & \lambda_2 & \lambda_1 \end{array} :y,z\right\}.
\end{align*}
The idea is now that we obtain three translations, given by $T_{\kappa_1,\lambda_1}$, $T_{\kappa_2,\lambda_2}$ and $T_{a_1,\lambda_1}$. The remaining translations may be found via compositions of the above listed symmetries with these three translations.

We outline our first translation as 
\[
T_{\kappa_1,\lambda_1} : \left\{ \begin{array}{c c c c} a_1 & a_2 & a_3 & \\ \kappa_1 & \kappa_2 & \lambda_1 & \lambda_2 \end{array}:y,z \right\} \to \left\{ \begin{array}{c c c c} a_2 & a_1 & a_3 & \\ q\kappa_1 & \kappa_2 & q\lambda_1 & \lambda_2 \end{array}:\tilde{y},\tilde{z} \right\}. 
\]
We know that 
\[
\det \tilde{A} = q \kappa_1\kappa_2(x-a_1)(x-a_2)(x-a_3)
\]
hence, the determinant is proportional to $x$. Using \eqref{firstterm}, the behavior at $x=0$ is 
\begin{align*}
R_{\kappa_1, \lambda_1} &\sim (\tilde{C}_0 + O(x))  \begin{pmatrix} e_{q,q\lambda_1}(x) & 0 \\ 0 & e_{q,\lambda_2}(x) \end{pmatrix}  \begin{pmatrix} e_{q,\lambda_1}(x) & 0 \\ 0 & e_{q,\lambda_2}(x) \end{pmatrix}^{-1} (C_0 + O(x))^{-1}, \\
& \sim \tilde{C}_0\begin{pmatrix} x & 0 \\ 0 & 1 \end{pmatrix} C_0^{-1} + O(x),
\end{align*}
hence, the degree of the entries of $R_{\kappa_1, \lambda_1}$ in $x$ are bounded from below by zero. We then obtain a representation of $R_{\kappa_1, \lambda_1}$ by expanding $\tilde{Y}_\infty Y_\infty^{-1}$ using \eqref{firstterms} up to polynomial terms, giving
\[
R_{\kappa_1, \lambda_1}(x) = \begin{pmatrix} \displaystyle \frac{q \tilde{y}-q \tilde{z}_1+q \tilde{\gamma}-q y+q z_1-q \gamma }{q-1}+x & -q w \\
\displaystyle -\frac{q \kappa _1 \tilde{\gamma}}{\kappa _2 \tilde{w}} & 1 \end{pmatrix}.
\]
We may now use the compatibility condition,
\[
R_{\kappa_1, \lambda_1}(q x) A(x)= \tilde{A}(x)R_{\kappa_1, \lambda_1}(x),
\]
to evaluate the relationship between $y$ and $z$ and $\tilde{y}$ and $\tilde{z}$, which is given by
\begin{align*}
\tilde{y} &= \frac{\frac{a_3 y^2 \kappa _1 \left(a_1 a_2 \kappa _2-q \lambda _1\right)}{a_1 a_2 \kappa _1-z \lambda _1}-\frac{a_3 q \kappa _1
   \left(y-a_1\right) \left(y-a_2\right)}{z}+q \lambda _1 \left(a_3-y\right)}{\frac{y^2 \kappa _1 \left(a_1 a_2 \kappa _2-q \lambda _1\right)}{a_1 a_2 \kappa _1-z \lambda _1}+y \kappa _2 \left(a_3-y\right)},\\
\tilde{z} &=-\frac{q y \kappa _1 \lambda _1 \left(\tilde{y}-a_1\right) \left(\tilde{y}-a_2\right)}{\kappa _2 \left(z \lambda _1 \left(a_3-y\right)
   \tilde{y}+a_1 a_2 a_3 \kappa _1 \left(y-\tilde{y}\right)\right)}.
\end{align*}
The inverse of this map is obtained by solving for $y$ and $z$ in terms of $\tilde{y}$ and $\tilde{z}$,
\begin{align*}
y &= \frac{\frac{a_3 \kappa _2 \tilde{y}^2 \left(a_1 a_2 \kappa _2-q \lambda _1\right)}{a_3 \kappa _2 \tilde{z}+q \lambda _1}+\frac{q \lambda _1
   \left(a_1-\tilde{y}\right) \left(a_2-\tilde{y}\right)}{\tilde{z}}+a_1 a_2 \kappa _2 \left(a_3-\tilde{y}\right)}{\frac{\kappa _2 \tilde{y}^2 \left(a_3 \kappa _2 \left(a_1+a_2-\tilde{y}\right)-q \lambda _1\right)}{a_3 \kappa _2 \tilde{z}+q \lambda
   _1}+\frac{\kappa _2 \tilde{y} \left(a_1-\tilde{y}\right) \left(a_2-\tilde{y}\right)}{\tilde{z}}},\\
z\tilde{z}&=\frac{\kappa _1 \left(q y \lambda _1 \left(\tilde{y}-a_1\right) \left(\tilde{y}-a_2\right)+a_1 a_2 a_3 \kappa _2 \left(y-\tilde{y}\right)
   \tilde{z}\right)}{\kappa _2 \lambda _1 \left(y-a_3\right) \tilde{y}}.
\end{align*}

Let us identify 
\[
T_{\kappa_2,\lambda_2} : \left\{ \begin{array}{c c c c} a_1 & a_2 & a_3 & \\ \kappa_1 & \kappa_2 & \lambda_1 & \lambda_2 \end{array}:y,z\right\} \to \left\{ \begin{array}{c c c c} a_2 & a_1 & a_3 & \\ \kappa_1 & q\kappa_2 & \lambda_1 & q \lambda_2 \end{array}:\tilde{y},\tilde{z} \right\},
\]
as a transformation of the form \eqref{transformation}. Using similar logic to the case above, the determinant of $R(x)$ is $x$, and we find that the degree of the entries of $R$ in $x$ are bounded from below by zero, hence, we may obtain a representation of $R_{\kappa_2,\lambda_2}$ by using \eqref{firstterms} to obtain an expansion of $\tilde{Y}_{\infty}Y_{\infty}^{-1}$ around $x=\infty$. Up to constant terms, this gives
\[
R_{\kappa_2,\lambda_2} = \begin{pmatrix} 1 & q \tilde{w} \\
 \displaystyle \frac{\gamma  \kappa _1}{w \kappa _2} & \displaystyle x+\frac{-q y+q \tilde{y}-q \alpha +q \tilde{\alpha} }{q-1}
\end{pmatrix}.
\]
We use the compatibility condition,
\[
R_{\kappa_1, \lambda_1}(q x) A(x)= \tilde{A}(x)R_{\kappa_1, \lambda_1}(x),
\]
to show
\begin{align*}
\tilde{y}y &= \frac{-\frac{y^2 \kappa _2 \lambda _1 \left(a_3 \kappa _1+\lambda _2\right)}{\kappa _1 \left(a_3 z \kappa _2+\lambda
   _1\right)}+\frac{\lambda _1 \left(y-a_1\right) \left(y-a_2\right)}{z}+a_1 a_2 \kappa _2 \left(a_3-y\right)}{\frac{\lambda _1 \left(a_3 \kappa _1 \kappa _2 \left(y-a_1\right) \left(y-a_2\right)+\lambda _1 \left(y \kappa _1+\lambda
   _2\right)\right)}{a_3 \kappa _1 z\left(a_3 z \kappa _2+\lambda _1\right)}-\frac{\lambda _1 \left(y \kappa _1+\lambda
   _2\right)}{a_3 \kappa _1z}},\\
\tilde{z} &= -\frac{y \kappa _1 \left(\tilde{y}-a_1\right) \left(\tilde{y}-a_2\right)}{\kappa _2 \tilde{y} \left(a_3 z-y (\tilde{y}+z)+y^2\right)+\lambda_1 (\tilde{y}-y)}.
\end{align*}
Conversely, one may solve for $y$ and $z$ in terms of $\tilde{y}$ and $\tilde{z}$, which is given by
\begin{align*}
y \tilde{y} &= \frac{\frac{a_3 \tilde{y}^2 \kappa _1 \left(\lambda _1-a_1 a_2 \kappa _2\right)}{\tilde{z} \lambda _1-a_1 a_2 \kappa _1}-\frac{a_3
   \kappa _1 \left(\tilde{y}-a_1\right) \left(\tilde{y}-a_2\right)}{\tilde{z}}+\lambda _1 \left(a_3-\tilde{y}\right)}{\frac{\tilde{y} \kappa _1 \left(\lambda _1-a_1 a_2 \kappa _2\right)}{\tilde{z} \lambda _1-a_1 a_2 \kappa _1}+\kappa _2
   \left(a_3-\tilde{y}\right)},\\ 
z \tilde{z} &= \frac{y \left(\kappa _1 \left(\tilde{y}-a_1\right) \left(\tilde{y}-a_2\right)+\tilde{y} \tilde{z} \kappa _2
   (y-\tilde{y})\right)+\tilde{z} \lambda _1 (\tilde{y}-y)}{\tilde{y}  \kappa _2 \left(y-a_3\right)}.
\end{align*}

We require one more transformation, before the remaining transformations may be derived via the use of suitable symmetries of $A(x)$. The required transformation is
\[
T_{a_1,\lambda_1} : \left\{ \begin{array}{c c c c} a_1 & a_2 & a_3 & \\ \kappa_1 & \kappa_2 & \lambda_1 & \lambda_2 \end{array}:y,z\right\} \to \left\{ \begin{array}{c c c c} qa_1 & a_2 & a_3 & \\ \kappa_1 & \kappa_2 & q\lambda_1 & \lambda_2 \end{array}:\tilde{y},\tilde{z} \right\}.
\]
Note that the determinant of $R_{a_1,\lambda_1}$ is given by
\[
\det R_{a_1,\lambda_1} = \frac{x}{x-qa_1}.
\]
We also note that by similar logic to the above two cases, the behavior around $x =0$ is constant plus terms of order $x$. The behavior at $x = \infty$ is also constant plus terms of order $1/x$. Hence, $R_{a_1,\lambda_1}$ is of the form
\begin{equation}\label{Rval}
R_{a_1,\lambda_1} = \frac{R_1x+ R_0}{x-qa_1}.
\end{equation}
Expanding \eqref{Rval} of around $x=\infty$, gives to the first two leading orders, 
\[
R_{a_1,\lambda_1} = I + \frac{R_0 + q a_1}{x} + O\left( \frac{1}{x^2}\right) ,
\]
whereas, expanding $\tilde{Y}_\infty Y_{\infty}^{-1}$ to the first two leading orders from \eqref{firstterm} gives
\[
R_{a_1,\lambda_1} = \begin{pmatrix}\displaystyle \frac{-q y-q \gamma +q z_1+q \tilde{y}+q \tilde{\gamma} -q \tilde{z}_1}{(q-1) x}+1 &\displaystyle  \frac{q \tilde{w}-q w}{x} \\
 \displaystyle \frac{\gamma  \kappa _1 \tilde{w}-w \kappa _1 \tilde{\gamma} }{w x \kappa _2 \tilde{w}} &\displaystyle  \frac{q \tilde{y}+q \tilde{\alpha}-q y-q \alpha}{(q-1) x}+1 \end{pmatrix} + O\left( \frac{1}{x^2}\right).
\]
Equating these gives
\[
R_{a_1,\lambda_1} = \frac{1}{x-qa_1}\begin{pmatrix}
\displaystyle x+\frac{q (y+\alpha-\tilde{y}-\tilde{\alpha})}{q-1} & q (\tilde{w}-w) \\
\displaystyle \frac{\kappa _1 (\gamma  \tilde{w}-w \tilde{\gamma} )}{w \kappa _2 \tilde{w}} & \displaystyle x+\frac{q \left(\tilde{y}+\tilde{\alpha}-y-\alpha \right)}{q-1} - qa_1
   \end{pmatrix}.
\]
Lastly, we use the compatibility,
\[
R_{\kappa_1, \lambda_1}(q x) A(x)= \tilde{A}(x)R_{\kappa_1, \lambda_1}(x),
\]
to show
\begin{align*}
\tilde{y} &=\frac{a_2 y \kappa _1 \left(y-a_1\right) \left(a_3 \kappa _2 \tilde{z}+q \lambda _1\right)}{\lambda _1 \left(q y \kappa _1
   \left(y-a_1\right)+z \kappa _2 \left(a_3-y\right) \tilde{z}\right)+a_2 a_3 \kappa _1 \kappa _2 \left(y-a_1\right) \tilde{z}},\\
\tilde{z} &= y \frac{\frac{\kappa _1 \left(a_1 q \lambda _1+a_2 a_3 z \kappa _2\right)}{a_1 y}+\frac{a_2 \kappa _1 \lambda _1 \left(z \kappa _2-q
   \kappa _1\right)}{a_1 a_2 \kappa _1-z \lambda _1}+\frac{a_2 \left(a_3-a_1\right) z^2 \kappa _1 \kappa _2 \lambda _1}{a_1
   \left(y-a_1\right) \left(a_1 a_2 \kappa _1-z \lambda _1\right)}}{\frac{\left(a_3-a_1\right) z^2 \kappa _2 \lambda _1^2}{a_1 \left(y-a_1\right) \left(a_1 a_2 \kappa _1-z \lambda
   _1\right)}+\frac{a_3 \kappa _2 \left(z \lambda _1-a_1 a_2 \kappa _1\right)}{a_1 y}+\frac{a_2 \kappa _1 \kappa _2 \left(a_2 a_3
   \kappa _1-z \lambda _1\right)}{a_1 a_2 \kappa _1-z \lambda _1}},
\end{align*}
and the inverse
\begin{align*}
y &= \tilde{z} \frac{\frac{a_2 a_3 \kappa _2^2 \tilde{y}^2 \left(a_3 \kappa _1+\lambda _1\right)}{a_3 \kappa _2 \tilde{z}+q \lambda _1}+\frac{\left(a_2-\tilde{y}\right)
   \left(a_2 a_3 \kappa _1 \kappa _2 \tilde{y}+q \lambda _1^2\right)}{\tilde{z}}+a_2 \kappa _2 \lambda _1 \left(a_3-\tilde{y}\right)}{\frac{q \kappa _1 \lambda _1 \left(a_2-\tilde{y}\right){}^2}{\tilde{z}}-\frac{q \kappa _2 \lambda _1 \tilde{y}^2 \left(a_3 \kappa _1+\lambda
   _1\right)}{a_3 \kappa _2 \tilde{z}+q \lambda _1}+a_2 \kappa _2 \left(a_2 a_3 \kappa _1+\lambda _1 \tilde{y}\right)},\\ 
z &= \frac{\kappa _1 \left(y-a_1\right) \left(q y \lambda _1 \left(\tilde{y}-a_2\right)+a_2 a_3 \kappa _2 \left(\tilde{y}-y\right)
   \tilde{z}\right)}{\kappa _2 \lambda _1 \left(y-a_3\right) \tilde{y} \tilde{z}}.
\end{align*}

The remaining generators for the translations of the connection data may be obtained by suitable conjugation of the symmetries, $r_0$, $r_1$ and $r_2$. Hence, we list the remaining generators as 
\begin{align*}
T_{a_2,\lambda_1} &= r_0 \circ T_{a_1,\lambda_1} \circ r_0,  \\  
T_{a_3,\lambda_1} &= r_1 \circ r_0 \circ T_{a_1,\lambda_1} \circ r_0 \circ r_1, \\
T_{a_2,\lambda_2} &= r_2 \circ r_0 \circ T_{a_1,\lambda_1} \circ r_0 \circ r_2,  \\ 
T_{a_3,\lambda_1} &= r_2 \circ r_1 \circ r_0 \circ T_{a_1,\lambda_1} \circ r_0 \circ r_1 \circ r_2, \\
T_{\kappa_1,\lambda_2} &= r_2 \circ T_{\kappa_1,\lambda_1} \circ r_2, \\ 
T_{\kappa_2,\lambda_1} &= r_2 \circ T_{\kappa_2,\lambda_2} \circ r_2, 
\end{align*}
which generates the full set of translational components of the connection data. The symmetries also give us a matrix representation of the for \eqref{transformation} for each translation. By \eqref{multrule}, we are able to express any connection preserving deformation as a product of $R$ matrices.

In this way, we may think of the evolution of $q$-$\mathrm{P}_{\mathrm{V}}$ as being given by 
\begin{align*}
q\textrm{-}\mathrm{P}_{\mathrm{V}} &= T_{a_1,\lambda_1} \circ T_{a_2,\lambda_2} \circ T_{\kappa_1,\lambda_1}^{-1} \circ T_{\kappa_2,\lambda_2}^{-1}.
\end{align*}
This theory also strongly suggests a factorization of the matrix that defines the evolution of $q$-$\mathrm{P}_{\mathrm{V}}$. 
It is also possible, via the transformations above, to express all the $R$ matrices in terms of untilded variables, hence, factorize any Schlesinger transformation into matrices specified above. 

The above information is sufficient in allowing us to form the lattice of translational B\"acklund transformations for \eqref{qPV:geo}. We note that the lattice, $A_4^{(1)}$, is spanned by basis elements with the following effect on the $b_i$:
\begin{align*} 
T_0(b_0) = qb_0, 					&\hspace{1cm}  T_1(b_0) = b_0 , 					&\hspace{1cm} T_2(b_0) = b_0, & \hspace{1cm} 
T_3(b_0) = b_0,  &\hspace{1cm} T_4(b_0) = \frac{b_0}{q}, \\
T_0(b_1) = \frac{b_1}{q},  &\hspace{1cm} T_1(b_1) = qb_1 , 					&\hspace{1cm} T_2(b_1) = b_1,  &\hspace{1cm}
T_3(b_1) =b_1,  &\hspace{1cm} T_4(b_1) = b_1, \\
T_0(b_2) = b_2,  					&\hspace{1cm} T_1(b_2) = \frac{b_2}{q},  &\hspace{1cm} T_2(b_2) = q b_2,  &\hspace{1cm}
T_3(b_2) =b_2,  &\hspace{1cm} T_4(b_2) = b_2, \\
T_0(b_3) = b_3,  					&\hspace{1cm} T_1(b_3) = b_3,  					&\hspace{1cm} T_2(b_3) = \frac{b_3}{q},  &\hspace{1cm}
T_3(b_3) =q b_4,  &\hspace{1cm} T_4(b_3) = b_3, \\
T_0(b_4) = b_4,  					&\hspace{1cm} T_1(b_4) = b_4,  					&\hspace{1cm} T_2(b_4) = b_4,  &\hspace{1cm}
T_3(b_4) = \frac{b_4}{q},  &\hspace{1cm} T_4(b_4) = qb_4. 
\end{align*} 
We identify the five elements of the basis of translational components of the B\"acklund transformation as being equivalent to the following connection preserving deformations:
\begin{align*}
T_0 : \left\{ \begin{array}{c c c c} a_1 & a_2 & a_3 & \\ \kappa_1 & \kappa_2 & \lambda_1 & \lambda_2 \end{array}:y,z\right\} &\to 
      \left\{ \begin{array}{c c c c} a_1 &qa_2 &qa_3 & \\ \kappa_1 & \kappa_2 &q\lambda_1 &q\lambda_2 \end{array}:\tilde{y},\tilde{z}\right\}, \\
T_1 : \left\{ \begin{array}{c c c c} a_1 & a_2 & a_3 & \\ \kappa_1 & \kappa_2 & \lambda_1 & \lambda_2 \end{array}:y,z\right\} &\to 
      \left\{ \begin{array}{c c c c}a_1/q&a_2/q^2&a_3/q& \\ \kappa_1 & q^2\kappa_2 & \lambda_1/q & \lambda_2/q \end{array}:\tilde{y},\tilde{z}\right\}, \\
T_2 : \left\{ \begin{array}{c c c c} a_1 & a_2 & a_3 & \\ \kappa_1 & \kappa_2 & \lambda_1 & \lambda_2 \end{array}:y,z\right\} &\to 
      \left\{ \begin{array}{c c c c} a_1 & a_2 & a_3 & \\ \kappa_1 & \kappa_2/q & \lambda_1 & \lambda_2/q \end{array}:\tilde{y},\tilde{z}\right\}, \\
T_3 : \left\{ \begin{array}{c c c c} a_1 & a_2 & a_3 & \\ \kappa_1 & \kappa_2 & \lambda_1 & \lambda_2 \end{array}:y,z\right\} &\to 
      \left\{ \begin{array}{c c c c} a_1 & a_2 & a_3 & \\ q\kappa_1 & \kappa_2 & \lambda_1 & q\lambda_2 \end{array}:\tilde{y},\tilde{z}\right\}, \\
T_4 : \left\{ \begin{array}{c c c c} a_1 & a_2 & a_3 & \\ \kappa_1 & \kappa_2 & \lambda_1 & \lambda_2 \end{array}:y,z\right\} &\to 
      \left\{ \begin{array}{c c c c} qa_1 & qa_2 & a_3 & \\ \kappa_1 & \kappa_2 & q\lambda_1 & q\lambda_2 \end{array}:\tilde{y},\tilde{z}\right\}, \\
T_0T_1T_2T_3T_4 : \left\{ \begin{array}{c c c c} a_1 & a_2 & a_3 & \\ \kappa_1 & \kappa_2 & \lambda_1 & \lambda_2 \end{array}:y,z\right\} &\to 
      \left\{ \begin{array}{c c c c} a_1 & a_2 & a_3 & \\ q\kappa_1 & q\kappa_2 & q\lambda_1 & q\lambda_2 \end{array}:y,z\right\}.
\end{align*}
The last translation is an identity on the space of parameters and of the Painlev\'e equation itself as discussed previously. Using \eqref{multrule} and the $R$ matrices that specify each connection preserving deformation, we may obtain a Lax pair, \eqref{qdiff} and \eqref{transformation}, for each translational B\"acklund transformation.

We note that $r_0$, $r_1$ and $r_2$ are Schlesinger manifestations of symmetries of the Painlev\'e equation, however, there are two symmetries that we have been unable incorporate into this theory. One symmetry seems to correspond to a switch of the asymptotic behaviors of single columns between solutions at $x=0$ and $x= \infty$, which we cannot, at present, see how this may be induced by the left multiplication of a rational matrix. The other is a Dynkin diagram automorphism. The full presentation of B\"acklund transformations may be derived from Sakai's' work \cite{Sakai:Rational}. Further investigation is warranted into how these other symmetries of the Painlev\'e equation may manifest themselves as symmetries, or perhaps more generally, as connection preserving deformations of the associated linear problem.

\section{$q$-differential equations for orthogonal polynomials}

We now move towards the second aim of this paper. This being a special orthogonal polynomial solution to the associated linear system above. We shall give a brief account of the theory of $q$-differential equations satisfied by $q$-orthogonal polynomials. We note that our approach, which is based on an extension of an approach of Laguerre \cite{OrmerodForresterWitte, Laguerre, Magnus1995, Shohat}, is one of many approaches found in the literature \cite{Biane, qladder}. Some have been known to produce discrete Painlev\'e equations \cite{Shohat}.

Given a sequence of moments, $\{\mu_k\}_{k=0}^{\infty}$, one may define the linear form on the space of polynomials so that
\begin{equation}\label{ortho:linearform}
L(x^k) = \mu_k.
\end{equation}
If the determinant,
\begin{equation}\label{ortho:delta}
\Delta_n = \det \begin{pmatrix} 
\mu_0 & \mu_1 & \ldots & \mu_{n-1} \\ 
\mu_1 & \mu_2 & \ldots & \mu_n\\
\vdots &   & \ddots & \vdots \\
\mu_{n-1} & \mu_n & \ldots & \mu_{2n-1}
\end{pmatrix},
\end{equation}
where $\Delta_0 = 1$, does not vanish for all $n \in \mathbb{N}$, then the polynomial sequence, $(p_n)_{n\in \mathbb{N}}$, where $p_n$ is of degree $n$, is uniquely defined by the condition
\[
L(p_i(x)p_j(x)) = \delta_{ij}.
\]
We also define
\begin{equation}\label{ortho:sigma}
\Sigma_n = \det \begin{pmatrix} 
\mu_0 & \mu_1 & \ldots & \mu_{n-2} & \mu_n \\ 
\mu_1 & \mu_2 & \ldots & \mu_{n-1}& \mu_{n+1} \\
\vdots &   & \ddots &  & \vdots &  \\
\mu_{n-1} & \mu_n & \ldots & \mu_{2n-2}& \mu_{2n}
\end{pmatrix},
\end{equation}
with initial values, $\Sigma_0 = 0$ and $\Sigma_1 = \mu_1$. In letting  $\{a_n\}_{n=1}^{\infty}$ and $\{b_n\}_{n=0}^{\infty}$ be 
\begin{subequations}
\begin{align}
\label{ortho:an} a_n^2 =& \frac{\Delta_{n-1}\Delta_{n+1}}{\Delta_n^2},\\
\label{ortho:bn} b_n =& \frac{\Sigma_{n+1}}{\Delta_{n+1}} - \frac{\Sigma_{n}}{\Delta_{n}},
\end{align}
\end{subequations}
then
\begin{equation}\label{ortho:3termrecurrence}
a_{n+1}p_{n+1}(x) = (x-b_n)p_n(x) - a_np_{n-1}(x).
\end{equation}
This is known as the three term recursion relation \cite{Szego}. 

The moments specify a Stieltjes function, or moment generating function, defined by
\begin{equation}\label{ortho:Stieltjes}
f(x) = \sum_{k=0}^{\infty} \frac{\mu_k}{x^{k+1}}.
\end{equation}
From this, we define the associated polynomials and associated functions via 
\begin{equation}\label{epsdef}
f(x)p_n(x) = \phi_{n-1}(x) + \epsilon_n(x),
\end{equation}
whereby multiplying \eqref{ortho:3termrecurrence} by $f$ reveals that $\{\epsilon_n\}_{n=0}^{\infty}$ and $\{\phi_{n-1}\}_{n=1}^{\infty}$ also satisfy \eqref{ortho:3termrecurrence}. We parameterize the orthogonal polynomials as
\begin{equation}\label{ortho:poly}
p_n(x) = \rho_n x^n + \rho_{1,n}x^{n-1} + \rho_{2,n}x^{n-2}+ \ldots,
\end{equation}
where
\begin{equation}
\label{ortho:rho} \rho_n^2 = \frac{\Delta_n}{\Delta_{n+1}}.
\end{equation}
Using \eqref{ortho:3termrecurrence}, we are able to provide a parameterization of the coefficients of $p_n$ and $\epsilon_n$ in terms of the $a_i$ and the $b_i$. The first few leading terms in the expansion of $p_n$ and $\epsilon_n$ around $x=\infty$ are
\begin{subequations}\label{ortho:largexexpansion}
\begin{align}
p_n =& \rho_n\left(x^n - x^{n-1}\sum_{i =0}^{n-1} b_i \right. \\ & \left. + x^{n-2}\left(\sum_{i =0}^{n-2} \sum_{j=i+1}^{n-1} b_i b_j - \sum_{i=1}^{n-1}a_i^2\right) + O(x^{n-3}) \right)\nonumber ,\\  
\epsilon_n =& \rho_n^{-1}\left(x^{-n-1} + x^{-n-2}\sum_{i =0}^n b_i \right. \\ & \left. +  x^{-n-3}\left(\sum_{i =0}^{n} \sum_{j=0}^{i} b_i
b_j + \sum_{i=1}^{n+1}a_i^2\right) +  O(x^{-n-4}) \right)\nonumber.
\end{align}
\end{subequations}
Identifying \eqref{ortho:largexexpansion} with \eqref{ortho:poly} reveals
\begin{align*}
a_n &= \frac{\rho_{n-1}}{\rho_n},\\
b_n &= \frac{\rho_{n,1}}{\rho_n} -  \frac{\rho_{n+1,1}}{\rho_{n+1}}.
\end{align*}
Another useful formula is obtained by equating $(fp_n)p_{n-1}$ with $(fp_{n-1})p_n$, giving
\begin{align}\label{pnident}
\phi_{n-1}p_{n-1} - \phi_{n-2}p_n = \epsilon_{n-1}p_n - \epsilon_np_{n-1} = \frac{1}{a_n}.
\end{align}

We now impose a structure that may be associated with $q$-orthogonal polynomial systems. We assume that there exists a recurrence relation for the moments, expressed in terms of the moment generating function as
\begin{equation}\label{ortho:recurrencemoments}
W(x) D_{q,x} f(x) = 2V(x)f(x) + U(x)
\end{equation}
where $W(x)$, $V(x)$ and $U(x)$ are polynomials. Assuming that the linear form may be expressed in terms of the Jackson $q$-integral
\begin{equation}\label{ortho:linearformweight}
L(f(x)) = \int_{a}^{b}w(x) f(x)d_{q}x
\end{equation}
for some weight function, $w(x)$, one may prove
\[
W(x)D_{q,x}w = 2V(x)w(x).
\]
We call $W(x)$ and $V(x)$ the spectral data polynomials \cite{OrmerodForresterWitte, Magnus1995}. We define
\[
\Psi_n(x) = \begin{pmatrix} p_n & \frac{\epsilon_n}{w} \\ p_{n-1} & \frac{\epsilon_{n-1}}{w} \end{pmatrix}.
\]
The following theorem may be seen as a consequence of the work of Magnus \cite{OrmerodForresterWitte, Magnus1995}.  

\begin{thm}\label{thm:Dqxpn}
The matrix, $\Psi_n$, is the solution to
\begin{equation}\label{Linear:thm}
D_{q,x}\Psi_n = \mathscr{L}_n \Psi_n,
\end{equation}
where
\begin{align}\label{Linear:Aform}
\mathscr{L}_n &= \frac{1}{(W-2x(1-q)V)}\begin{pmatrix}
\Omega_n - V & -a_n \Theta_n \\
a_n \Theta_{n-1} & \Omega_{n-1} - V - (x-b_{n-1})\Theta_{n-1}
\end{pmatrix},
\end{align}
and $\Theta_n$ and $\Omega_n$ are polynomials of bounded degree, specified by
\begin{subequations}\label{thm:defThetaOmega}
\begin{eqnarray}
\label{thm:defTheta}\Theta_n &=& W (\epsilon_n D_{q,x} p_n - p_n D_{q,x} \epsilon_n) + 2V\epsilon_n(p_n - x(1-q)D_{q,x}p_n),\\
\label{thm:defOmega}\Omega_n &=& a_n W (\epsilon_{n-1} D_{q,x} p_n - p_{n-1} D_{q,x} \epsilon_n) \\
&& \hspace{.5cm} + a_nV (p_n\epsilon_{n-1} + p_{n-1}\epsilon_n) - 2V x(1-q) a_n \epsilon_{n-1} D_{q,x} p_n. \nonumber
\end{eqnarray}
\end{subequations}
\end{thm}

For a detailed proof, we refer to \cite{OrmerodForresterWitte}. Using \eqref{ortho:largexexpansion}, the degree of $\Omega_n$ and $\Phi_n$ are bounded by
\begin{subequations}\label{eq3:degthetaomega}
\begin{align}
\label{eq3:degtheta}\deg_x{\Theta_n} &\leq \max(\deg_x W-1, \deg_x V -2 ,0),\\
\label{eq3:degomega}\deg_x{\Omega_n} &\leq \max(\deg_x W, \deg_x V -1 ,0).
\end{align}
\end{subequations}
Furthermore, by substituting the known expansion of $\epsilon_n$ and $p_n$, given by \eqref{ortho:largexexpansion}, into \eqref{thm:defTheta} and \eqref{thm:defOmega}, we may evaluate all terms of $\Omega_n$ and $\Theta_n$ in terms of the $a_i$'s, $b_i$'s and $\rho_i$'s. 

Using theorem \ref{thm:Dqxpn}, we define the matrix
\begin{equation}\label{ortho:Aform}
L_n = I - x(1-q)\mathscr{L}_n,
\end{equation} 
so that we express the $q$-differential equation as a system of linear $q$-difference equations of the familiar (see \cite{Adams, qBirkhoff, Carmichael, Sakai:qP6, Murata2009, Sakai:Garnier, Sauloy, vanderPut}) form 
\begin{equation}\label{ortho:linearx}
\Psi_n(qx) = L_n(x)\Psi_n(x),
\end{equation}
where $L_n(x)$ is rational, and given by
\begin{equation}\label{Lnform}
L_n = \begin{pmatrix} {\displaystyle \frac{(q-1) x \left(\Omega _n+V\right)+W}{W-2x (1-q)V}} &{\displaystyle -\frac{(q-1) x a_n \Theta _n}{W-2x (1-q)V}} \\
{\displaystyle \frac{(q-1) x a_n \Theta _{n-1}}{W-2x (1-q)V} } & {\displaystyle \frac{(q-1) x \left(\Theta _{n-1} (b_{n-1}-x)+\Omega _{n-1}+V\right)+W}{W-2x (1-q)V}}
\end{pmatrix}.
\end{equation}
We notice that by \eqref{pnident}, the determinant of $\Psi_n$ is
\begin{equation}\label{DetPsi}
\det \Psi_n = \frac{p_n\epsilon_{n-1} - p_{n-1}\epsilon_n}{w} = \frac{1}{a_nw}.
\end{equation}
This useful identity gives us that
\begin{equation}\label{ortho:con}
\det L_n = \frac{w(x)}{w(qx)}= \frac{W}{W-2x(1-q)V}.
\end{equation}

In addition to the $q$-differential equation in $x$, \eqref{ortho:3termrecurrence} is equivalent to 
\begin{equation}\label{ortho:linearn}
\Psi_{n+1}(x) = M_n(x)\Psi_n(x),
\end{equation}
where 
\[
M_n(x) = \begin{pmatrix}
\frac{x-b_n}{a_{n+1}} & -\frac{a_n}{a_{n+1}} \\ 1 & 0 \end{pmatrix}.
\]
The compatibility condition between \eqref{ortho:linearx} and \eqref{ortho:linearn} is
\[
L_{n+1}(x) M_n(x) = M_n(qx)L_n(x),
\]
which implies
\begin{subequations}
\begin{align}
\label{Frued1}\Omega _{n+1} \left(x-b_n\right)+ \Omega _n \left(q x-b_n\right)-x(1-q) V &= a_{n+1}^2 \Theta _{n+1}- a_n^2 \Theta _{n-1} \\
\label{Frued2}\Theta _n \left(q x-b_n\right)+\Theta _{n-1} \left(b_{n-1}-x\right)&=\Omega_{n+1} - \Omega_{n-1}.
\end{align}
\end{subequations}
These equations are $q$-differential analogues of the Freud-Laguerre equations \cite{Magnus1995}. We should also mention that there is a ladder operator \cite{Ladder} approach to finding the $q$-difference equation satisfied by the orthogonal polynomial system \cite{qladder}. 

In addition to the differential system in $x$, we note that there is additional structure to be obtained from the methods above. We shall explore the relation between a system of polynomials polynomials, $\{p_n\}$, associated with a weight function, $w$, to a system of polynomials, $\{\tilde{p}_n\}$, with associated weight function, $\tilde{w}$, which is, a priori, a rational multiple of $w$. In a previous study \cite{OrmerodForresterWitte}, this was considered a $t$ shift, however, we find that the concept, although very fruitful in continuous systems \cite{Magnus1995}, is not congruent with current trends in discrete Painlev\'e equations \cite{Noumi, Sakai:Rational}. We wish to take a slightly different approach to the material presented in  \cite{OrmerodForresterWitte}, and to caste the deformation theory in a similar manner to \S 2. 

The first deviation from previous material lies in an equivalent regularity condition, which we take to be
\[
Rw = S\tilde{w},
\]
where $R$ and $S$ are polynomials in $x$. This implies a recurrence in for the moments, governed by 
\begin{equation}\label{ftilde}
R\tilde{f} = Sf+T,
\end{equation}
where $T$ is a polynomial in $x$. These types of conditions, namely \eqref{ftilde}, are specify the relation between the sequence $\{\mu_k\}_{k=0}^\infty$ and the sequence $\{\tilde{\mu}_k\}_{k=0}^{\infty}$ \cite{Bonan1, Bonan2, OrmerodForresterWitte, Magnus1995}. We find that in order for this system to be consistent, the two was of calculating $D_{q,x}\tilde{w}$ and $D_{q,x}\tilde{f}$ must agree, imposing the constraints
\begin{align*}
\frac{2\tilde{V}(x)}{\tilde{W}(x)} \frac{S(x)}{R(x)} &= \frac{S(qx)}{R(qx)}\frac{2V(x)}{W(x)} + D_{q,x} \frac{S(x)}{R(x)},\\
\frac{2\tilde{V}(x)}{\tilde{W}(x)}\frac{T(x)}{R(x)} &=  \frac{S(qx)}{R(qx)}\frac{U(x)}{W(x)} + D_{q,x} \frac{T(x)}{R(x)}.
\end{align*}
The following theorem specifies the relation between $\Psi_n$ and $\tilde{\Psi}_n$.

\begin{thm}
The matrix, $\Psi_n$, is related to $\tilde{\Psi}_n$ via
\begin{align}
\tilde{\Psi}_n = \mathscr{R}_n \Psi_n,
\end{align}
where
\begin{align}\label{Rdef}
\mathscr{R}_n = \frac{1}{S}\begin{pmatrix}\displaystyle \Xi_n & - a_n \Phi_n \\
\displaystyle a_n\Phi_{n-1} & \Xi_{n-1}- (x-b_{n-1})\Phi_{n-1} 
\end{pmatrix},
\end{align}
and $\Xi_n$ and $\Phi_n$ are polynomials of bounded degree, specified by
\begin{subequations}
\begin{align}
\label{defPhi}\Phi_n &=  S\epsilon_n\tilde{p}_n - R\tilde{\epsilon}_n p_n,\\
\label{defXi}\Xi_n &= a_n S \epsilon_{n-1}\tilde{p}_n - a_nR\tilde{\epsilon}_np_{n-1}.
\end{align}
\end{subequations}
\end{thm}

\begin{proof}
Using \eqref{ftilde} and \eqref{epsdef}, 
\[
R\left(\frac{\tilde{\phi}_{n-1}}{\tilde{p}_n} +  \frac{\tilde{\epsilon}_n}{\tilde{p}_n}\right) = S\left( \frac{\phi_{n-1}}{p_n} + \frac{\epsilon_{n}}{p_n}\right) + T,
\]
hence, we may define the polynomial, $\Phi_n$, to be 
\[
Rp_n\tilde{\phi}_{n-1} - S\phi_{n-1} \tilde{p}_n - Tp_n\tilde{p_n} = S\epsilon_n\tilde{p}_n - R \tilde{\epsilon}_np_n = \Phi_n.
\]
By \eqref{pnident}, 
\[
(a_np_{n-1}\phi_{n-1} - a_np_n\phi_{n-2})\Phi_n = Rp_n\tilde{\phi}_{n-1} - S\phi_{n-1} \tilde{p}_n - Tp_n\tilde{p_n},
\]
hence, we define $\Xi_n$ via
\begin{subequations}
\begin{align}
\label{Xiline1}\phi_{n-1}p_n \Xi_n &= \phi_{n-1} \left[ a_np_{n-1} \Phi_n + S\tilde{p}_n \right],  \\
\label{Xiline2} &= p_n \left[ a_n\phi_{n-2}\Phi_n  + R\tilde{\phi}_{n-1} -T\tilde{p}_n \right].
\end{align}
\end{subequations}
The second line, \eqref{Xiline2}, is equivalent to the definition of $\tilde{p}_n$ from \eqref{Rdef}. Hence, to show $\tilde{p}_n$ is congruent with \eqref{Rdef}, we need only show that $\Xi_n$ is polynomial and that it is given by \eqref{defXi}. To show $\Xi$ is given by \eqref{defXi}, dividing \eqref{Xiline1} by $\phi_{n-1}p_n$ reveals
\[
\Xi_n = \frac{a_np_{n-1} \Phi_n}{p_n} + \frac{S\tilde{p}_n}{p_n},
\]
whereby using \eqref{defPhi} implies
\[
\Xi_n = \frac{Sa_np_{n-1}\epsilon_n\tilde{p}_n}{p_n} - R a_n\tilde{\epsilon}_np_{n-1} + \frac{S\tilde{p}_n}{p_n}.
\]
Using \eqref{pnident}, reduces the above to \eqref{defXi}. To see $\Xi_n$ is a polynomial, dividing \eqref{Xiline2} by $p_n\phi_{n-1}$ shows
\begin{align}
\Xi_n &= \frac{a_n\phi_{n-2}\Phi_n}{\phi_{n-1}}  + \frac{R\tilde{\phi}_{n-1}}{\phi_{n-1}} -\frac{T\tilde{p}_n}{\phi_{n-1}},\\
&= \frac{a_n\phi_{n-2}}{\phi_{n-1}} \left[Rp_n\tilde{\phi}_{n-1} - S\phi_{n-1} \tilde{p}_n - Tp_n\tilde{p_n} \right]  + \frac{R\tilde{\phi}_{n-1}}{\phi_{n-1}} -\frac{T\tilde{p}_n}{\phi_{n-1}},
\end{align} 
where use of \eqref{pnident} reduces this to
\[
\Xi_n = a_nRp_{n-1}\tilde{\phi}_{n-1} - a_nS\phi_{n-2}\tilde{p}_n - a_nTp_{n-1}\tilde{p}_n.
\]
To see that $\epsilon_n/w$ satisfies the same $q$-difference equation, note that
\begin{align*}
\tilde{f}\tilde{p}_n &= \tilde{\phi}_{n-1} + \tilde{\epsilon}_n\\
\tilde{f}\left[ \frac{\Xi_n}{S} p_n  - \frac{a_n\Phi_np_{n-1}}{S}\right]&= \left[\frac{S}{R}f + \frac{T}{R} \right]\left[ \frac{\Xi_n}{S} p_n  - \frac{a_n\Phi_np_{n-1}}{S}\right]\\
&= \frac{\Xi_n}{R} (\phi_{n-1} + \epsilon_n) + \frac{a_nT\Phi_n}{RS}(\phi_{n-2} + \epsilon_{n-1}).
\end{align*}
Seeing as \eqref{Xiline2} defines $\tilde{\phi}_{n-1}$ in terms of $\tilde{p}_n$, which is given by \eqref{Rdef}, the remaining terms give
\[
\tilde{\epsilon}_n = \frac{\Xi_n}{R}\epsilon_n - \frac{a_n\Phi_n}{R}\epsilon_{n-1},
\]
hence, $\epsilon_n/w$ satisfies 
\[
\frac{\tilde{\epsilon}_n}{\tilde{w}_n} = \frac{\Xi_n}{S}\frac{\epsilon_n}{w} - \frac{a_n\Phi_n}{S}\frac{\epsilon_{n-1}}{w}.
\]
The second line, which defines $\tilde{p}_{n-1}$ and $\tilde{\epsilon}_{n-1}/\tilde{w}$, is obtained from \eqref{ortho:3termrecurrence}.
\end{proof}

In the same manner as theorem \ref{thm:Dqxpn}, the terms $p_n$ and $\epsilon_n$ appear together, hence, the degree of the polynomials, $\Xi_n$ and $\Phi_n$, have bounds on their degree in $x$ given by
\begin{subequations}
\begin{align}
\deg_x \Xi_n \leq& \max(\deg_x R-1 , \deg_x S-1,0), \\
\deg_x \Phi_n \leq& \max(\deg_x R-2 , \deg_x S). 
\end{align}
\end{subequations}
This means that $\Xi_n$ and $\Phi_n$ may be computed explicitly in terms of the $a_i$'s and $b_i$'s using \eqref{ortho:largexexpansion}.

The above system is subject to compatibility with all the previously established recurrences. We note that \eqref{DetPsi} implies
\begin{equation}\label{letR}
\det \mathscr{R}_n = \frac{wa_n}{\tilde{w}\tilde{a}_n} = \frac{a_nR}{\tilde{a}_nS}.
\end{equation}
We shall explore some options for particular choices of deformation. This above is a generalization of the theorem that appeared in \cite{OrmerodForresterWitte}, however, this type of deformation has a deeper history\cite{Christofell, Ismail:book, Uvarov}. 

\section{The Big $q$-Laguerre Polynomials.}

The big $q$-Laguerre polynomials and $q$-$\mathrm{P}_{\mathrm{V}}$ are degenerations of the big $q$-Jacobi polynomials and $q$-$\mathrm{P}_{\mathrm{VI}}$ respectively. We shall establish that a simple generalization of the big $q$-Laguerre polynomials forms a column vector solution to a special case of the associated linear problem for $q$-$\mathrm{P}_{\mathrm{V}}$. The monic versions of the big $q$-Laguerre polynomials defined as
\begin{align}\label{Laguerredef}
P_n =& {}_3 \phi_2 \left(\begin{array}{c |} q^{-n},0,x\\ aq, bq \end{array} \hspace{.1cm} q;q  \right) \\
=& \frac{1}{(b^{-1}q^{-n};q)_n} {}_2\phi_1 \left( \begin{array}{c|} q^{-n}, aqx^{-1} \\ aq \end{array} \hspace{.1cm} q;\frac{x}{b}\right),\nonumber
\end{align}
are orthogonal with respect to the linear form
\begin{equation}\label{bigqLagurreLinearForm}
L(f) = \int_{bq}^{aq}w(x,a,b;q)f(x)\mathrm{d}_q x,
\end{equation}
where
\begin{equation}\label{weight}
w(x,a,b;q) = \frac{\left( \frac{x}{a},\frac{x}{b};q\right)_\infty}{(x;q)_\infty }.
\end{equation}
We generalize these polynomials in a way that preserves the structure of the linear problem these polynomials satisfy. This contrasts the approach in \cite{OrmerodForresterWitte} where the degree of the recursion relation the moments satisfy is altered in a manner that introduces a parameter, $t$. However, in this framework, we have no $t$ parameter, only translations on a lattice represented by bi-rational transformations. 

From the orthogonal polynomial viewpoint, there are some simple generalizations of these polynomials that allow us to extrapolate extra variables, for example, if we include a variable associated with a scaling of $x$. By considering polynomials specified by \eqref{ortho:linearform} and a weight function of
\begin{equation}\label{genweight}
w(x) = \frac{x^{\sigma}\left( \frac{x}{a_1},\frac{x}{a_3} ;q\right)_{\infty}}{\left(\frac{x}{a_2};q\right)_{\infty}}.
\end{equation}
We choose a support in which the endpoints generalize \eqref{bigqLagurreLinearForm}, hence we write the orthogonality condition explicitly as
\begin{equation}\label{P5:orthogonality}
L(p_i(x)p_j(x)) = \int_{qa_1}^{qa_3} w(x)p_i(x)p_j(x) \mathrm{d}_q x = \delta_{ij}.
\end{equation}
This condition specifies the polynomials completely. We identify the moment integral,
\[
\mu_k = \int_{qa_1}^{qa_2} x^k w(x)\mathrm{d}_q x = \int_{qa_1}^{qa_2} x^{k+\sigma} \frac{\left( \frac{x}{a_1},\frac{x}{a_2} ; q\right)_{\infty}}{\left( \frac{x}{a_3}; q\right)_{\infty}} \mathrm{d}_q x,
\]
as a special case of \eqref{hyperint} where $c = 0$. Hence, we may write
\begin{align}
\label{moments}\mu_k =& \frac{(1-q)(q;q)_{\infty}}{(q^{\sigma+k + 1};q)_\infty} \left[(qa_3)^{\sigma+k+1} {}_2\phi_1 \left( \begin{array}{c |} \frac{a_3}{a_1} , q^{\sigma+k+1} \\ 0 \end{array} \hspace{.1cm}q;\frac{qa_2}{a_1} \right) \right. \\
&-\left. (qa_1)^{\sigma+k+1} {}_2\phi_1 \left( \begin{array}{c |} \frac{a_3}{a_2} , q^{\sigma+k+1} \\ 0 \end{array} \hspace{.1cm}q;\frac{qa_1}{a_2} \right) \right].\nonumber
\end{align}
Note that when $\sigma$ is a negative integer, \eqref{moments} truncates, leaving a rational function.

We shall denote these generalized polynomials be
\[
p_n(x) = p_n\left(\begin{array}{c} a_1, a_2\\a_3, q^\sigma \end{array} , x \right),
\]
where the relation to the big $q$-Laguerre is given by
\[
P_n(x) = \frac{1}{\rho_n}p_n\left(\begin{array}{c} a, 1\\ b, 1 \end{array} , x \right).
\]
A simple calculation reveals
\[
D_{q,x} w(x) = \frac{ (x-a_1)(x-a_2)a_3 + q^{\sigma} a_1a_2(x-a_3)}{x(1-q)(x-a_1)(x-a_2)a_3}w(x),
\]
which specifies spectral data polynomials as
\begin{subequations}\label{P5:spectraldata}
\begin{align}
W &= x(1-q)(x-a_1)(x-a_3)a_2,\\
2V &= (x-a_1)(x-a_3)a_2 + q^{\sigma}a_1a_3(x-a_2).
\end{align}
\end{subequations}
The use of \eqref{thm:defThetaOmega} and \eqref{ortho:largexexpansion} gives us that $\deg \Omega_n = 3$ and $\deg \Theta_n = 2$, hence we let
\begin{subequations}
\begin{align}
\Omega_n =& \frac{a_2}{2}x^2 - \frac{1}{2}x(a_2a_3 + a_1(a_2+a_3q^{\sigma}(2q^n-1))) \\
& + \frac{1}{2} q^{-n-1} \left(a_1 a_3 q^n \left(a_2 q \left(q^{\sigma }+1\right)-2 q^{n+\sigma } \left((1-q) \Gamma_n +a_2 q\right)\right)+2 a_2 a_n^2\right) \nonumber \\
\Theta_n =& \frac{a_2}{q^{n+1}}x + q^{n+\sigma}a_1a_3 - \frac{a_2}{q^{2+n}}\left(q\Gamma_{n} + qa_1 + qa_3-\Gamma_{n+1} \right)
\end{align}
\end{subequations}
where 
\[
\Gamma_n = -\frac{\rho_{n,1}}{\rho_n} = \sum_{i=0}^{n-1} b_i.
\]
Given the above spectral data polynomials and values of $\Omega_n$ and $\Theta_n$, may take $L_n$ to be of the form
\[
L_n = \frac{L_{0,n} + L_{1,n}x + L_{2,n}x^2}{x-a_3}.
\]
Using \eqref{ortho:con}, we find that
\begin{equation}\label{P5:detA}
\det L_n = \frac{a_3(x-a_1)(x-a_2)}{a_1a_2q^{\sigma}(x-a_2)}.
\end{equation}
We now preform a gauge transformation that will relate the system of $q$-difference equations satisfied by big $q$-Laguerre polynomials to a special case of the associated linear problem for $q$-$\mathrm{P}_\mathrm{V}$. The gauge transformation will be of the form
\[
Y_n(x)= \frac{1}{\left(\frac{x}{a_2};q \right)_{\infty}}\begin{pmatrix}\frac{1}{\rho_n}& 0 \\ 0 & \frac{1}{\rho_{n-1}} \end{pmatrix} \Psi_n,
\]
so that $Y_n$ satisfies 
\[
Y(qx) = A_n(x) Y_n(x),
\]
where $A_n(x)$ is of the form
\begin{equation}
A_n = A_{0,n} + A_{1,n}x + A_{2,n}x^2.
\end{equation}
The determinant of $A_n$, from this transformation, is given by
\begin{equation}\label{detAp}
\det A_n = -\frac{(x-a_1)(x-a_2)(x-a_3)}{q^{\sigma}a_1a_2a_3},
\end{equation}

It is a simple exercise write the matrix representation
\begin{align*}
A_{2,n} =& \begin{pmatrix} 0 & 0 \\ 0 & \kappa_2 \end{pmatrix},\\
A_{1,n} =& \begin{pmatrix} \kappa_1 & \displaystyle \frac{\kappa_2a_n^2}{q} \\ -\kappa_2 & \displaystyle \frac{\kappa _2 (1-q) \Gamma_n}{q} -a_1 - a_3 \end{pmatrix},
\end{align*}
where
\begin{align}
\kappa_1 &= -\frac{q^n}{a_2},\\
\kappa_2 &= \frac{1}{a_1a_3q^{n+\sigma}}.
\end{align}
Before stating the value of $A_{0,n}$, it is useful to evaluate the $A_{0,n}$'s eigenvalues, $\lambda_{1,n}$ and $\lambda_{2,n}$. This will allow us to simplify the resulting expression for $A_{0,n}$. Firstly, note that the determinant of $A_{0,n}$, as found by letting $x=0$ in \eqref{detAp}, is
\[
\lambda_{1,n} \lambda_{2,n} = q^{-\sigma}.
\]
Since $W$ is divisible by $x$, the trace of $A_0$ in terms of $W$, $V$, $\Omega_n$ and $\Theta_n$ is
\[
\mathrm{tr} A_{0,n} = \lambda_{1,n} + \lambda_{2,n} = \Omega_n(0)  + \Theta_{n-1}(0)b_{n-1} + \Omega_{n-1}+ V(0).
\]
However, by examining $x=0$ in \eqref{Frued2}, we this expression to see
\[
(\lambda_{1,n} + \lambda_{2,n}) - (\lambda_{1,n+1} +\lambda_{2,n+1}) = 0.
\]
Hence, we may determine the eigenvalues of $A_{0,n}$ by considering the initial values of $A_{0,n}$, say $A_{0,1}$, in terms of the $\mu_k$. Using this and the recurrence from \eqref{ortho:recurrencemoments}, 
\[
a_2\mu _{k+1}=q \left(a_1 a_2 a_3 \mu _{k-1} q^{k+\sigma +1}-a_1 a_3 \mu _k q^{k+\sigma +1}-a_1 a_2 a_3 q \mu _{k-1}+a_1 a_2
   \mu _k+a_2 a_3 \mu _k\right)
\]
for $k > 1$ tells us
\[
\mathrm{tr} A_{0,n} = 1+ q^{-\sigma},
\]
hence, 
\begin{align}
\lambda_{1,n} &= 1,\\
\lambda_{2,n} &= q^{-\sigma}.
\end{align}

\begin{rem}
We suppose the linear problem for the orthogonal polynomial system, \eqref{ortho:linearx}, admits solutions of the form given by \eqref{expansionYinf} and \eqref{expansionY0}. Since, by construction, \eqref{ortho:linearx} admits a polynomial solution of degree $n$, for this to coincide with a solution of the form \eqref{expansionYinf}, one $\kappa_i$ must be $q^{-n}$ as $e_{q^{-n},q}(x) \propto x^n$. Similarly, either $\lambda_1$ or $\lambda_2$ must be $1$ to admit polynomial solutions as $e_{1,q}(x)$ is constant in $x$. 
\end{rem}

This allows us to write $A_{0,n}$ as
\[
A_{0,n} = \begin{pmatrix}
 r_n & a_n^2 s_n \\
 -\frac{\left(r_n-\lambda_{1,n}\right) \left( r_n-\lambda_{2,n}\right)}{a_n^2 s_n} & \lambda_{1,n} + \lambda_{2,n} -r_n
\end{pmatrix},
\]
where we have introduced variables
\begin{align*}
r_n &= \frac{(q-1) \kappa _1 \Gamma _n}{q}-\frac{\kappa _2 a_n^2+a_2 q \kappa _1}{q}, \\
s_n & =\frac{\kappa _2 \left(b_n-(q-1) \Gamma _n - qa_1-qa_3  \right)}{q^2}-\kappa _1,
\end{align*}
for convenience.

We now state the change of variables that relates the coefficient matrix for $Y_n(x)$ with the coefficient matrix for $Y(x)$. We note that in making this correspondence, the values of $w$, $y$ and $z$ must depend on $n$, hence we write $w = w_n$, $y = y_n$ and $z = z_n$. These values, written in terms of the $a_i$ and $b_i$, are
\begin{subequations}
\begin{align}
y_n &=a_1+a_3-\frac{b_n+\Gamma _n}{q}+\Gamma _n+\frac{q \kappa _1}{\kappa _2}, \\
w_n &= \frac{a_n^2}{q}, \\
\label{z1}z_{1,n} &=a_1-a_2+a_3-\frac{b_n+2 \Gamma _n}{q}+2 \Gamma _n+\frac{q \kappa _1}{\kappa _2}-\frac{w \kappa _2}{\kappa _1}, \\
\label{P5:zconstraint}z_{1,n}z_{2,n} &= (y_n - a_1)(y_n-a_2)(y_n - a_3),
\end{align}
\end{subequations}
with a corresponding factorization of \eqref{P5:zconstraint} given by
\begin{align}
z_{1,n} &= \frac{\left(a_1-y_n\right) \left(a_2-y_n\right)}{z_n},\\
z_{2,n} &= \left(y_n-a_3\right) z_n.
\end{align}
The above identification means $z_n$ in terms of $\Gamma_n$, $a_n$ and $b_n$. We have identified that the $q$-difference equation satisfied by the Big $q$-Laguerre polynomials is a special case of the associated linear problem for $q$-$\mathrm{P}_\mathrm{V}$. Conversely, we may express the quantities, $a_n^2$, $b_n$ and $\Gamma_n$ as
\begin{align}
a_n^2 &= q w_n\\
b_n &= \frac{q \left(z_n \left(\kappa _2 \kappa _1 \left(a_1+a_2+a_3-2 y_n\right)+q \kappa _1^2+w_n \kappa _2^2\right)\right)}{z_n \kappa _1 \kappa _2} + z_{1,n}\\
\Gamma_n &= \frac{q \left(\frac{w_n \kappa _2}{\kappa _1}-z_{1,n}-1\right)}{q-1}.
\end{align}

As naturally suggested by the previous sections, there is one choice of deformation which shall give rise to $q$-$\mathrm{P}_{\mathrm{V}}$. We consider the relation between polynomials associated with \eqref{genweight} and polynomials associated with
\begin{equation}\label{genweighttilde}
\tilde{w}(x) = \frac{x^{\sigma}\left( \frac{x}{qa_1},\frac{x}{a_3} ;q\right)_{\infty}}{\left(\frac{x}{qa_2};q\right)_{\infty}}.
\end{equation}
This gives us spectral data polynomials of
\begin{subequations}
\begin{align}
R(x) &= a_1(x-qa_2),\\
S(x) &= a_2(x-qa_1).
\end{align}
\end{subequations}
This data is sufficient to deduce the associated $\Phi_n$ and $\Xi_n$ as
\begin{align}
\Xi_n &= \frac{a_2 x \tilde{\rho}_n}{\rho _n}+\frac{a_2 \tilde{\rho}_n \left(\Gamma _n -\tilde{\Gamma}_n-qa_1\right)}{\rho _n}-\frac{a_1 \rho
   _n}{\tilde{\rho}_n},\\
\Phi_n &= \frac{a_2 \tilde{\rho} _n}{\rho _n}-\frac{a_1 \rho _n}{\tilde{\rho}_n}.
\end{align}
This reveals that the transformed equation is of the form
\[
R_n(x) = \frac{R_{1,n}x + R_{2,n}x}{(x-qa_1)(x-qa_2)},
\]
where $R_1$ is diagonal. However, we note that by construction, the leading behavior of a column solution is given by
\[
\frac{1}{\left( \frac{x}{qa_2};q\right)_\infty} \begin{pmatrix} x^n + O(x^{n-1}) \\ x^{n-1} + O(x^{n-2}) \end{pmatrix} =\ \frac{ R(x) }{\left( \frac{x}{a_2};q\right)_\infty} \begin{pmatrix} x^n + O(x^{n-1}) \\ x^{n-1} + O(x^{n-2}) \end{pmatrix} .
\]
This forces $R_{1,n}$ to be $qa_2I$. The remaining entries are simply determined from the correspondence between $\Gamma_n$, $a_n^2$ and $b_n$ with $y_n$, $z_n$ and $w_n$. We may divide $R_n(x)$ by $qa_2$ to obtain an equivalent representation to that of \eqref{Rmat} and \eqref{P5:B0}. The evolution of $y_n$ and $z_n$ are then determined by theorem \ref{thm:qP5}.

We note that this deformation of polynomials is not constrained to the above deformation. We may decompose this deformation into a group of deformations in much the same manner. This implies a certain structure inherent in the weight. Each deformation is equivalent to a deformation of the connection data. This is outlined in table \ref{tableconnection}.

\begin{table}[!ht]
\begin{tabular}{|c|c|c|c|c|}
\hline
 & $R$ & $S$ & $\tilde{w}$ & $\tilde{M}$ \\ \hline \hline
$T_\sigma$ & $1$ & $x$ &  $\displaystyle \frac{x^{\sigma+1}\left( \frac{x}{a_1},\frac{x}{a_3} ;q\right)_{\infty}}{\left(\frac{x}{a_2};q\right)_{\infty}}$ &
$\left\{ \begin{array}{c c c c} a_1 & a_2 & a_3 & \\ \kappa_1 & \frac{\kappa_2}{q} & \frac{\lambda_1}{q} & \lambda_2 \end{array}\right\}$ \\ \hline
$T_{a_1}$ & $1$ & $\left( 1- \frac{x}{qa_1}\right)$  &  $\displaystyle \frac{x^{\sigma}\left( \frac{x}{qa_1},\frac{x}{a_3} ;q\right)_{\infty}}{\left(\frac{x}{a_2};q\right)_{\infty}}$ & $\left\{ \begin{array}{c c c c} q a_1 & a_2 & a_3 & \\ \kappa_1 & \frac{\kappa_2}{q} & \lambda_1 & \lambda_2 \end{array}\right\}$ \\ \hline
$T_{a_3}$ & $1$ &$\left( 1- \frac{x}{qa_3}\right)$ &  $\displaystyle \frac{x^{\sigma}\left( \frac{x}{a_1},\frac{x}{qa_3} ;q\right)_{\infty}}{\left(\frac{x}{a_2};q\right)_{\infty}}$ & $\left\{ \begin{array}{c c c c} a_1 & a_2 & qa_3 & \\ \kappa_1 & \frac{\kappa_2}{q} & \lambda_1 & \lambda_2 \end{array}\right\}$ \\ \hline
$T_{a_2}$ & $\left( 1- \frac{x}{qa_2}\right)$ & $1$ & $\displaystyle \frac{x^{\sigma}\left( \frac{x}{a_1},\frac{x}{a_3} ;q\right)_{\infty}}{\left(\frac{x}{qa_2};q\right)_{\infty}}$ & $\left\{ \begin{array}{c c c c} a_1 & qa_2 & a_3 & \\ \frac{\kappa_1}{q} & \kappa_2 & \lambda_1 & \lambda_2 \end{array}\right\}$ \\ \hline
\end{tabular} 
\caption{A table outlining the affect of changing the weight by some rational factor on the associated monodromy data.}
\label{tableconnection}
\end{table}
The specification of the polynomials, $R$ and $S$, completely specifies the evolution. Note that we have one more translation available to us, namely the translational component that defines the $n \to n+ 1$ evolution, specified by \eqref{ortho:3termrecurrence}. This is not given by the specification of a set of polynomials, $R$ and $S$, but has the effect of keeping the weight constant, and 
\[
T_n : \left\{ \begin{array}{c c c c} a_1 & a_2 & a_3 & \\ \kappa_1 & \kappa_2 & \lambda_1 & \lambda_2 \end{array}\right\} \to \left\{ \begin{array}{c c c c} a_1 & a_2 & a_3 & \\ q\kappa_1 & \frac{\kappa_2}{q} & \lambda_1 & \lambda_2 \end{array}\right\}.
\]
Furthermore, by specifying that $\lambda_2 = 1$ remains constant, this condition is equivalent to the quotient group seen in \S 2. We now may specify the correspondence between the transformations of \S 2 and the deformations of the polynomial:
\begin{align*}
T_0 &= T_{a_2} \circ T_{a_3}, \\ 
T_1 &= T_{a_1} \circ T_{a_1}^{-1} \circ T_{a_2}^{-2} \circ T_{a_3}^{-1} \circ T_{n}^{-1},\\
T_2 &= T_{\sigma} \circ T_n^{-1},\\
T_3 &= T_{a_1}\circ T_{a_2},\\
T_4 &= T_0^{-1}\circ T_1^{-1} \circ T_2^{-1} \circ T_3^{-1} = T_{a_1} \circ T_{a_2}.
\end{align*}
We also note that each of these combinations of weight deformations and recurrences that appear on the right hand sides must be a copy evolution of $q$-$\mathrm{P}_{\mathrm{V}}$ since the Dynkin diagram automorphism may be used to express any of translational components in terms of just one translational component in $\tilde{W}\left(A_4^{(1)}\right)$. 

Finally, we remark that we may use the above framework to solve $q$-$\mathrm{P}_\mathrm{V}$ in terms of Hankel determinants
\begin{subequations}
\begin{align}
y_n &=a_1+a_3-\frac{\Sigma_{n+1}}{q\Delta_{n+1}}+\frac{\Sigma_n}{\Delta_n}+\frac{q \kappa _1}{\kappa _2}, \\
z_n &=\frac{q \kappa _1 \Delta _n^2 \left(y_n-a_1\right) \left(y_n-a_2\right)}{\kappa _1 \Delta _n \left( q\Delta _n( y_n-a_2) +(q-1) \Sigma
   _n\right)-\kappa _2 \Delta _{n-1} \Delta _{n+1}}, 
\end{align}
\end{subequations}
where $\Delta_n$ and $\Sigma_n$ are given by \eqref{ortho:delta} and \eqref{ortho:sigma} repsectively, $\mu_k$ is given by \eqref{moments}.

\section{Discussion}

The paper hopefully sheds some light on the correspondence between the symmetries of the $q$-Painlev\'e equations and their associated linear problems. This work also shows that there is an intimate connection between the basic hypergeometric and rational solutions of the $q$-Painlev\'e equations and the associated linear functions admitting polynomial solutions. We have briefly touched on the asymptotic form of solutions of irregular systems of $q$-difference equations, yet this work exhibits all the characteristics required to pursue a similar analysis for other associated linear problems. For those working on the derivation of $q$-Painlev\'e equations from an orthogonal polynomial system, this work should provide some insight into what combination of transformations of the weight and recurrences in $n$ give rise to known discrete Painlev\'e equations.

\section*{Acknowledgments}

I would like to thank Peter Forrester for useful discussions and suggestions.

\bibliography{refs}{}
\bibliographystyle{abbrv}

\end{document}